\documentclass[12pt]{amsart}
\usepackage{mathrsfs}
\usepackage{amssymb}
\usepackage{amsbsy}
\newtheorem{theorem}{Theorem}[section]
\newtheorem{lemma}[theorem]{Lemma}
\newtheorem{corollary}[theorem]{Corollary}

\theoremstyle{definition}
\newtheorem{definition}[theorem]{Definition}

\theoremstyle{remark}

\numberwithin{equation}{section}

\begin{document}

\begin{center}
{\Large \bf  Orlicz-Legendre Ellipsoids}
\end{center}

\vskip 20pt

\begin{center}
{{\bf Du~~Zou,~~~~Ge~~Xiong}\\~~ \\ Department of Mathematics, Shanghai University, Shanghai, 200444, PR China}
\end{center}

\vskip 10pt

\footnotetext{E-mail address: xiongge@shu.edu.cn}
\footnotetext{Research of the authors was supported by NSFC No. 11471206.}

\begin{center}
\begin{minipage}{14cm}
{\bf Abstract}  The Orlicz-Legendre ellipsoids, which are in the framework of
emerging dual Orlicz Brunn-Minkowski theory, are introduced for the first time.
They are in some sense dual to the recently found Orlicz-John ellipsoids,
and have largely generalized the classical Legendre ellipsoid of inertia. Several new
affine isoperimetric inequalities are established. The connection between the characterization
of Orlicz-Legendre ellipsoids and isotropy of measures is demonstrated.



\vskip 10pt
\end{minipage}
\end{center}

\vskip 25pt
\section{\bf Introduction}
\vskip 10pt

Corresponding to each body in Euclidean $n$-space $\mathbb R^n$,
there is a unique ellipsoid with the following property: The moment
of inertia of the ellipsoid and the moment of inertia of the body
are the same about \emph{ every} $1$-dimensional subspace of
$\mathbb R^n$. This ellipsoid is called the  \emph{ Legendre ellipsoid}
of the body. The Legendre ellipsoid is a well-known
concept from classical mechanics, and is closely related with  the
long-sanding unsolved maximal slicing problem. See, e.g.,
Lindenstrauss and Milman \cite{LindenstraussMilman}, and Milman and Pajor \cite{Milman}.

The Legendre ellipsoid is an object in the \emph{dual}
Brunn-Minkowski theory, which was originated by Lutwak \cite{Lutwak}
and achieved great developments since 1980s. See, e.g.,
\cite{BourgainZhang,Gardner1,Gardner2,Gardner4,Lutwak3,Lutwak4,Zhang1,Zhang2}.
It is remarkable that  for each convex body (compact convex subset
with non-empty interior) $K$ in $\mathbb R^n$, Lutwak, Yang and
Zhang \cite{LYZ1} introduced a new ellipsoid by using the notion of
$L_{2}$-curvature, which is now called the \emph{LYZ ellipsoid} and
is precisely the dual analogue of the Legendre ellipsoid.

Following LYZ \cite{LYZ1}, we write $\Gamma_{2} K$ and $\Gamma_{-2} K$ for
the Legendre ellipsoid and LYZ ellipsoid, respectively. In \cite{LYZ4},
LYZ extended the domain of $\Gamma_{-2}$ to star-shaped sets and showed
the relationship between the two ellipsoids: If $K$ is a star-shaped set, then
$\Gamma_{-2} K\subset\Gamma_2 K,$
with equality if and only if $K$ is an ellipsoid centered at the origin. This inclusion
is the geometric analogue of one of the basic inequalities in information theory - the
\emph{Cramer-Rao inequality}.  When viewed as suitably normalized matrix-valued operators on
the space of convex bodies, it was proved by Ludwig \cite{Ludwig1} that the Legendre
ellipsoid and the LYZ ellipsoid are the only linearly invariant operators that
satisfy the inclusion-exclusion principle.
The Legendre ellipsoid has also applications in Finsler geometry \cite{Matveev}.

In the geometry of convex bodies, many extremal problems of an affine nature often have
ellipsoids as extremal bodies.  Besides the above mentioned Legendre ellipsoid and
LYZ ellipsoid,  the \emph{John ellipsoid} ${\rm J}K$ \cite{John2} and the  \emph{L$\ddot{o}$wner ellipsoid}
${\rm L} K$ are of fundamental importance. Since the object considered in this paper
is \emph{dual} to the John ellipsoid, in what follows, we recall the John ellipsoid in detail.

Associated with each convex body $K$ in  $\mathbb R^n$,
its John ellipsoid ${\rm J}K$ is the unique ellipsoid
of maximal volume contained in $K$. The John ellipsoid has many applications in
convex geometry, functional analysis, PDEs, etc. Particularly, by combining the isotropic
characterization of the John ellipsoid and the celebrated Brascamp-Lieb inequality, it has
powerful effect on attacking  reverse isoperimetric problems. See, e.g., \cite{Ball1, Ball2,  LYZ5,LYZ6,LYZ7, Schuster}.

Since 2005, the family of John ellipsoid has expanded rapidly,
and experienced the $L_p$ stage \cite{LYZ6} and the very recent Orlicz stage \cite{ZouXiong}.
It is interesting that with the expansion of the family,  several  ellipsoids,
including the LYZ ellipsoid, are found to be close relatives of the John ellipsoid. We do a bit review on this point.

Motivated by the study of geometry of $L_p$ Brunn-Minkowski theory
(See, e.g., \cite{L1, L2, LYZ2}),  LYZ \cite{LYZ6} introduced a family of ellipsoids,
called the $L_p$ John ellipsoids ${\rm E}_p K$, $p>0$. It is striking that
the bodies ${\rm E}_p K$ form a \emph {spectrum} linking  several fundamental objects in convex geometry:
If the John point of $K$, i.e., the center of ${\rm J}K,$ is at the origin, then ${\rm E}_\infty K$ is
precisely the classical John ellipsoid ${\rm J}K$.
The $L_2$ John ellipsoid ${\rm E}_2 K$ is just the LYZ ellipsoid. The $L_1$ John ellipsoid ${\rm E}_1 K$ is
the so-called \emph{Petty ellipsoid}. The volume-normalized Petty
ellipsoid  is obtained by minimizing the surface area of $K$ under
${\rm SL}(n)$ transformations of $K$  (\cite{Gian1,Petty1}).

Throughout this paper, we consider convex $\varphi:[0,\infty)\to[0,\infty)$,
that is strictly increasing and satisfies $\varphi(0)=0$.  Along the line of extension,
the authors of this paper originally introduced the  \emph{ Orlicz-John ellipsoids} \cite{ZouXiong}
${\rm E}_\varphi K$ for each convex
body $K$ with the origin in its interior, in the framework of booming Orlicz Brunn-Minkowski theory
(See, e.g., \cite{ Gardner5, Gardner6, HLYZ, Ludwig2,LYZ8, LYZ9}).
The new Orlicz-John ellipsoids ${\rm E}_\varphi K$ generalize LYZ's
$L_p$ John ellipsoids ${\rm E}_p K$ to the Orlicz setting, analogous
to the way that Orlicz norms  \cite{Rao} generalize $L_p$
norms. Indeed, If $\varphi(t)=t^p$, $1\le p<\infty$, then ${\rm E}_\varphi K$ precisely
turns to the $L_p$ John ellipsoid ${\rm E}_p K$.
If $p\to\infty$,  then ${\rm E}_{\varphi^p} K$ approaches to ${\rm E}_\infty K$.

The \emph{L$\ddot{o}$wner ellipsoid} ${\rm L} K$ is the unique ellipsoid of minimal volume
containing $K$, which is investigated widely  in the field of convex geometry and  local theory of Banach
spaces. We refer to, e.g., \cite{Ball1, Ball2,ChenZhouYang, Gian1,
Gruber1, Gruber2, Gruber3, Gruber4, John2, Klartag, Lewis, LindenstraussMilman,Ludwig3, Pisier, Zhu}.

As LYZ  \cite{LYZ4}  pointed out, there is in
fact a ``dictionary" correspondence between the Brunn-Minkowski theory and its
dual. In retrospect, the John ellipsoid, LYZ ellipsoid and Petty ellipsoid are objects within
the Brunn-Minkowski theory; while the Legendre ellipsoid and L$\rm \ddot{o}$wner ellipsoid are objects within the
dual Brunn-Minkowski theory. Along the idea of dictionary  relation, we are tempted to consider
the naturally posed problem:  What is the dual analogue of the newly found Orlicz-John ellipsoid?

One of the main task in this paper is to demonstrate this existence of such
a \emph{dual} analogue of Orlicz-John ellipsoid. Incidentally,  it precisely acts
as the \emph{spectrum} linking the Legendre ellipsoid and L$\rm \ddot{o}$wner ellipsoid.
So, this paper is a sequel of \cite{ZouXiong}.

For star bodies $K,L$ in $\mathbb R^n$, define the
\emph{normalized dual Orlicz mixed volume} ${\bar{\tilde V}}_\varphi(K,L)$ of $K$ and $L$
with respect to $\varphi$ by
\[{\bar{\tilde V}}_\varphi(K,L) = {\varphi ^{ - 1}}\left( {\int_{{S^{n - 1}}} {\varphi \left( {\frac{{{\rho _K}}}{{{\rho _L}}}} \right)dV_K^*} } \right).\]
Here, $S^{n-1}$ is the unit sphere in $\mathbb R^n$; $\rho_K$ and
$\rho_E$ are the radial functions of $K$ and $L$, respectively;
$V_K^*$ is the normalized dual conical measure of $K$, defined by
\[dV_K^*=\frac{\rho_K^n}{nV(K)}dS, \]
where $S$ is the spherical Lebesgue measure on $S^{n-1}$.

Enlightened by our work on Orlicz-John ellipsoids
\cite{ZouXiong}, we focus on

{\bf Problem ${\tilde S}_\varphi$.} \emph{Suppose  $K$ is a star
body in $\mathbb R^n$. Find an  ellipsoid $E$, amongst all
origin-symmetric ellipsoids, which solves the following  constrained
minimization problem:}
\[\mathop {{\rm{min}}}\limits_E V(E)\quad{\rm{subject}}\;{\rm{to}}\quad {\bar{\tilde V}}_\varphi (K,E) \le 1.\]

In Section 4, we prove that there exists a unique ellipsoid which
solves the above minimization problem. It is called the \emph{Orlicz-Legendre ellipsoid}
of $K$ with respect to $\varphi$, and denoted by
${\rm L}_\varphi K$. If $\varphi(t)=t^2$, then ${\rm L}_\varphi K$ is precisely
the Legendre ellipsoid $\Gamma_2 K$.

It is interesting that the Orlicz-Legendre ellipsoid mirrors the Orlicz-John ellipsoid.

Similar to the important property of Orlicz-John ellipsoid  ${\rm E}_\varphi K$,
in Section 5 we show that the Orlicz-Legendre ellipsoid ${\rm L}_\varphi K$ is
jointly continuous in $\varphi$ and $K$ . In Section 6, it is
proved that as $p\to\infty$,  ${\rm L}_{\varphi^{p}} K$ approaches to a common ellipsoid ${\rm L}_\infty K$,  the
unique ellipsoid of minimal volume containing $K$. This insight throws light on a connection between
Orlicz-Legendre ellipsoids and the L${\rm \ddot{o}}$wner ellipsoid.

In Section 7, we  establish a characterization of Orlicz-Legendre ellipsoids,
which is closely related to the isotropy of measures.

In general,  Orlicz-Legendre ellipsoids ${\rm L}_\varphi K$ do not  contain $K$. In Section 8, we  prove  that: If $K$
is a star body (about the origin) in $\mathbb R^n$, then
\[ V({\rm L}_\varphi K)\ge V(K),\]
with equality if and only if $K$ is an ellipsoid centered at the origin.

If $\varphi(t)=t^2$, it reduces to the celebrated
inequality: $V(\Gamma_2 K)\ge V(K),$ which goes back to Blaschke \cite{Blaschke},
John \cite{John1}, Milman and Pajor \cite{Milman}, Petty \cite{Petty2}, and also LYZ  \cite{LYZ1}.

\vskip 25pt
\section{\bf Preliminaries}
\vskip 10pt

\subsection{Notations}

The setting will be the  Euclidean $n$-space $\mathbb R^n$. As usual,
$x\cdot y$ denotes the standard inner product of $x$ and $y$ in
$\mathbb R^n$, and $V$ denotes  the $n$-dimensional volume.

In addition to its denoting absolute value, without confusion we  often use $|\cdot|$ to
denote the standard Euclidean norm, on occasion the total mass of a
measure,  and the absolute value of the determinant of an $n\times n$ matrix.

For a continuous real function $f$ defined on $S^{n-1}$,
write $\|f\|_\infty$ for the $L_\infty$ norm of $f$.
Let $\mathscr{L}^n$ denote the space of linear operators
from $\mathbb R^n$ to $\mathbb R^n$.
For $T\in \mathscr{L}^n$,  $T^t$ and $\|T\|$ denote the transpose and norm of $T$, respectively.

A finite positive Borel measure $\mu$ on $S^{n-1}$ is said to be
\emph{isotropic} if
\[ \frac{n}{|\mu|}\int_{{S^{n - 1}}} {{{(u \cdot v)}^2}d\mu (u)}= 1,\quad {\rm for\; all}\; v\in S^{n-1}.\]

For nonzero $x\in\mathbb R^n$, the notation $x\otimes x$ represents the rank 1
linear operator on $\mathbb R^n$ that takes $y$ to $(x\cdot y) x$.
It immediately gives
\[{\mathop{\rm tr}\nolimits} (x \otimes x) = |x{|^2}.\]
Equivalently, $\mu$ is isotropic if
\[\frac{n}{|\mu|}\int_{{S^{n - 1}}} {u \otimes ud\mu (u)}  =  {I_n},\]
where $I_n$ denotes the identity operator on $\mathbb R^n$.
For more information on the isotropy of measures, we refer to \cite{BR, Gian1, Gian2, Milman}.

\subsection{Orlicz norms}

Throughout this paper, $\Phi$ denotes the class of convex functions $\varphi:
[0,\infty)\to [0,\infty)$, that are strictly increasing and satisfy $\varphi(0)=0$.

We say  a sequence $\{\varphi_i\}_{i\in\mathbb N}\subset\Phi$ is
such that ${\varphi _i} \to {\varphi _0}\in\Phi$, provided
\[|{\varphi _i} - {\varphi _0}{|_I}:  =  \mathop {\max }\limits_{t \in I} |{\varphi _i}(t) - {\varphi _0}(t)| \to0,\]
for each compact interval $I\subset [0,\infty)$.

Let $\mu$ be a finite positive Borel measure on $S^{n-1}$.  For a
continuous function $f: S^{n-1}\to [0,\infty)$, the \emph{Orlicz norm} $ {\left\| f: \mu \right\|_\varphi }$ of $f$,
is defined by
\[{\left\| f:\mu \right\|_\varphi } = \inf \left\{ {\lambda  > 0:\frac{1}{{|\mu |}}\int_{{S^{n - 1}}} {\varphi \left( {\frac{f}{\lambda }} \right)d\mu }  \le \varphi (1)} \right\}.\]

If $\varphi(t)=t^p$, $1\le p<\infty$, then $\|f:\mu\|_\varphi$ is just the classical $L_p$ norm.
According to the context,  without confusion  we  write $\|f\|_\varphi$  for $\|f:\mu\|_\varphi$.

Lemma 2.1 was previously proved in \cite{HLYZ}, which will be used frequently.
\begin{lemma}\label{lem 2.1}
Suppose $\mu$ is a finite positive Borel measure on $S^{n-1}$ and
the function $f: S^{n-1}\to [0,\infty)$ is continuous and such that
$\mu(\{{f\ne 0}\})>0$. Then the function
\[\psi (\lambda ):  = \int_{{S^{n - 1}}} {\varphi \left( {\frac{f}{\lambda }} \right) d\mu },\;\;\lambda\in (0,\infty),\]
has the following properties:

\noindent (1) $\psi$ is continuous and strictly decreasing in $(0,\infty)$;

\noindent (2) $ \mathop {\lim }\limits_{\lambda  \to {0^ + }} \psi (\lambda) = \infty $;

\noindent (3) $\mathop {\lim\limits _{\lambda  \to \infty }}\psi (\lambda ) = 0$;

\noindent (4) $0 < {\psi ^{ - 1}}( a ) < \infty  $ for each $a\in (0,\infty)$.
\end{lemma}
Consequently, the Orlicz norm ${\left\|f\right\|}_\varphi$ is strictly positive. Moreover,
\[{\left\| f \right\|_\varphi } = {\lambda _0}\quad \Longleftrightarrow\quad \frac{1}{{|\mu |}}\int_{{S^{n - 1}}} {\varphi \left( {\frac{f}{{{\lambda _0}}}} \right)d\mu }  = \varphi (1).\]

\subsection{Convex bodies and star bodies}

The \emph{support function} $h_K$ of a compact convex set $K$ in $\mathbb R^n$
is defined  by
\[ h_K(x)=\max\{x\cdot y: y\in K\},\quad {\rm for}\; x\in\mathbb R^n.\]
For $T\in{\rm GL}(n)$,  the support function of the image $TK=\{Tx:x\in K\}$ is given by
\[{h_{TK}}(x) = {h_K}({T^t}x). \]

As usual, a \emph{body} is a compact set with non-empty interior.
Write $\mathcal K^n_o$ for the class of convex bodies in $\mathbb R^n$ that
contain the origin in their interiors.  $\mathcal K^n_o$ is often equipped with the \emph{Hausdorff metric} $\delta_H$, which is defined  by
\[ {\delta _H}({K_1},{K_2}) =   {\max }\{ |h_{K_1}(u) - h_{K_2}(u)|: u\in S^{n-1}\},\quad {\rm for}\; K_1, K_2\in\mathcal K^n_o.\]
That is
\[\delta_H(K_1,K_2)=\|h_{K_1}-h_{K_2}\|_\infty.\]

Next, we turn to some basics on star bodies.

A set $K\subseteq\mathbb R^n$ is \emph{star-shaped}, if $\lambda x\in K$ for  $\forall (\lambda, x)\in [0,1]\times K$.
For a  non-empty, compact and star-shaped set  $K$ in $\mathbb R^n$,  its
\emph{radial function} $\rho_K$  is defined by
\[ \rho_K(x)=\sup\{\lambda\ge 0: \lambda x\in K\},\quad {\rm for}\; x\in \mathbb R^n\setminus \{o\}. \]
It is easily seen  that $\rho_K$ is homogeneous of degree $-1$. For $T\in {\rm GL}(n)$, we obviously have
\begin{equation}\label{(2.1)} \rho_{TK}(x)=\rho_K(T^{-1}x).
\end{equation}

A  star-shaped set $K$ is called a \emph{star body} about the origin $o$, if $o\in {\rm int}K$,
and its radial function $\rho_K$ is continuous on $S^{n-1}$.
Write $\mathcal{S}^n_o$ for the class of  star bodies about the origin $o$ in $\mathbb R^n$.
$\mathcal{S}^n_o$ is often equipped with the \emph{dual Hausdorff metric} $\tilde\delta_H$, which is defined by
\[ \tilde\delta_H(K_1,K_2)=\max\{|\rho_{K_1}(u)-\rho_{K_2}(u)|: u\in S^{n-1}\}, \quad {\rm for}\; K_1, K_2\in\mathcal{S}^n_o.\]
That is,
\[ \tilde\delta_H(K_1,K_2)=\|\rho_{K_1}-\rho_{K_2}\|_\infty. \]

The \emph{dual conical  measure} $\tilde V_K$, of a star body $K\in \mathcal{S}^n_o$,
is a Borel measure on $S^{n-1}$ defined by
\[d\tilde V_K=\frac{\rho_K^n}{n}dS. \]
It is convenient to use its normalization $ V^*_K$, given by $ V^*_K=\frac{\tilde V_K}{V(K)}$.
Observe that $V^*_K$ was firstly introduced by LYZ \cite{LYZ9} to define Orlicz centroid bodies.
Note that the dual conical measure differs from the
cone-volume measure (See, e.g., \cite{BLYZ1, BLYZ2, HLYZ, Henk,  LYZ8, Stancu, Xiong1}), but both are outgrowth from the cone measure
(See, e.g. \cite{ BGMN, Gromov, Naor}).

Note that for each Borel subset $\omega\subseteq S^{n-1}$, we also have
\[\tilde V_K(\omega)=V\left(K\cap\{su : s\ge0\; {\rm and }\; u\in\omega\}\right). \]
Thus, it follows that
\begin{equation}\label{(2.2)}
{\tilde V}_{TK}(\omega)={\tilde V}_K(\langle T^{-1}\omega\rangle), \quad {\rm for}\; T\in {\rm SL}(n),
\end{equation}
where $\langle T^{-1}\omega\rangle=\{\frac{T^{-1}u}{|T^{-1}u|}: u\in\omega\}$.

For $K\in\mathcal{K}^n_o$, its  \emph{ polar body} $K^*$ of $K$ is defined by
\[K^*=\{x\in\mathbb R^n: x\cdot y\le 1,\; {\rm for\; }\;  y\in K\}. \]

For $K\in\mathcal{K}^n_o$, we have
\begin{equation}\label{polarity1}
\rho_{K^*}(u)=\frac{1}{h_K(u)}\quad {\rm and}\quad h_{K^*}(u)=\frac{1}{\rho_K(u)}, \quad  {\rm for}\; u\in S^{n-1},
\end{equation}
and
\begin{equation}\label{polarity2}
(TK)^*=T^{-t}K^*, \quad {\rm for}\; T\in {\rm GL}(n).
\end{equation}

\subsection{Ellipsoids and linear operators}
Throughout, $\mathcal E^n $ is used exclusively to denote the class
of $n$-dimensional origin-symmetric ellipsoids in $\mathbb R^n$.

For $E\in \mathcal E^n$, let $d(E)$ denote its maximal
principal radius. Two facts are in order. First,  $T\in\mathscr{L}^n$ is
non-degenerated, if and only if the ellipsoid $TB$ is non-degenerated. Second, for $T\in\mathscr{L}^n$, since
\begin{align*}
\left\| T \right\|  = \mathop {\max }\limits_{u \in {S^{n - 1}}} |Tu| = \mathop {\max }\limits_{u \in {S^{n - 1}}} |{T^t}u| =\left\| {{T^t}} \right\|,
\end{align*}
it follows that
\begin{align*}
d(TB)= \mathop {\max }\limits_{u \in {S^{n - 1}}} {h_{TB}(u)}=\mathop {\max }\limits_{u \in {S^{n - 1}}} |{T^t}u| =\mathop {\max }\limits_{u \in {S^{n - 1}}} |Tu|
 = \mathop {\max }\limits_{u \in{S^{n - 1}}} {h_{{T^t}B}(u)}  = d({T^t}B).
\end{align*}

Let
\[d_n(T_1,T_2)=\|T_1-T_2\|,\quad {\rm for}\;  T_1, T_2\in \mathscr{L}^n.\]
Then the metric space $(\mathscr{L}^n, d_n)$ is complete. Since
$\mathscr{L}^n$ is of finite dimension, a set in $(\mathscr{L}^n, d_n)$ is compact,
if and only if it is bounded and closed.

We conclude this section with three lemmas, which will be used in
Sections 4 - 6. For their proofs, we refer to  Appendix A.
\begin{lemma}\label{lem 2.2}
Suppose $\{T_j\}_{j\in\mathbb N}\subset {\rm SL}(n)$. Then
\[\|T_j\|\to\infty\quad \Longleftrightarrow\quad \|T_j^{-1}\|\to\infty. \]
Thus,  $\{T_j\}_{j\in\mathbb N}$ is bounded from above, if
and only if $\{T_j^{-1}\}_{j\in\mathbb N}$ is bounded from above.
\end{lemma}

\begin{lemma}\label{lem 2.3}
Suppose $\{T_j\}_{j\in\mathbb N}\subset {\rm SL}(n)$, and $T_j\to T_0 \in {\rm SL}(n)$ with respect to  $d_n$. Then

\noindent (1) $T_j^t B\to T_0^tB$ with respect to $\delta_H$.

\noindent (2) $T_j^{-1}\to T_0^{-1}$ with respect to $d_n$.

\noindent (3) $T_jB\to T_0 B$ with respect to ${\tilde \delta}_H$.
\end{lemma}

\begin{lemma}\label{lem 2.4}
Suppose $E_0\in\mathcal{E}^n$,  $\{E_j\}_{j\in\mathbb N}\subset \mathcal{E}^n$ and $V(E_j)=a$, $\forall j \in\mathbb N$,  $a>0$. Then $E_j\to E_0$ with respect to
$\delta_H$, if and only if $E_j\to E_0$ with respect to ${\tilde\delta}_H$.
\end{lemma}

\vskip 25pt
\section{\bf Dual Orlicz mixed volumes}
\vskip 10pt
In order to define Orlicz-Legendre ellipsoids,  we  make some necessary preparations.

\begin{definition}\label{def 3.1}
Suppose $K,L\in\mathcal{S}^n_o$ and $\varphi\in\Phi$. The geometric quantity
\[{\tilde V}_{\varphi}(K,L) :=  \int_{{S^{n - 1}}} {\varphi \left( {\frac{{{\rho _K}}}{{{\rho _L}}}} \right)d{\tilde V}_K} \]
is called the \emph{dual Orlicz mixed volume} of $K$ and $L$ with respect to $\varphi$. The quantity
\[{\bar{\tilde V}}_{\varphi}(K,L) := {\varphi ^{ - 1}}\left( {\frac{{{{\tilde V}_\varphi }(K,L)}}{{V(K)}}} \right) = {\varphi ^{ - 1}}\left( {\int_{{S^{n - 1}}} {\varphi \left( {\frac{{{\rho _K}}}{{{\rho _L}}}} \right)d V_K^*} } \right)\]
is called the \emph{normalized dual Orlicz mixed volume} of $K$ and $L$ with respect to $\varphi$.
\end{definition}

Obviously, $\tilde V_{\varphi}(K,K)=\varphi(1)V(K)$, and ${\bar{\tilde V}}(K,K)=1$.

If $\varphi(t)=t^p$, $1\le p<\infty$, then ${\tilde V}_{\varphi}(K,L)$ reduces to the classical dual mixed volume
\[ {\tilde V}_{-p}(K,L)=\int_{S^{n-1}}{\left(\frac{\rho_K}{\rho_L} \right)^pd{\tilde V}_K},\]
and ${\bar{\tilde V}}_{\varphi}(K,L)$ reduces to normalized dual mixed volume \cite{Yu}
\[{ \bar{\tilde V}}_{-p}(K,L): = {\left[ {\frac{{{{\tilde V}_p}(K,L)}}{{V(K)}}} \right]^{\frac{1}{p}}} = {\left( {\int_{{S^{n - 1}}} {{{\left( {\frac{{{\rho _K}}}{{{\rho _L}}}} \right)}^p}dV_K^*} } \right)^{\frac{1}{p}}}.\]

\begin{lemma}\label{lem 3.2}
Suppose $K, L\in\mathcal{S}^n_o$ and $\varphi\in\Phi$. Then

\noindent (1) $\tilde V_\varphi(TK,L)= |T| \tilde V_\varphi(K,T^{-1}L)$,  for  $ T \in{\rm GL}(n)$.

\noindent (2) ${\bar{\tilde V}}_\varphi(TK,L)= {\bar{\tilde V}}_\varphi(K,T^{-1}L)$,  for  $T\in{\rm GL}(n)$.

\noindent (3) ${\bar{\tilde V}}_\varphi(\lambda K,L)={\bar{\tilde V}}_\varphi(K,\lambda^{-1}L)$, for  $\lambda>0$.
\end{lemma}

\begin{proof}
Suppose $T\in {\rm GL}(n)$. For $u\in S^{n-1}$, let $\langle T^{-1}\rangle=T^{-1}u/|T^{-1}u|$.
From Definition \ref{def 3.1}, (\ref{(2.1)}) and (\ref{(2.2)}), it follows that
\begin{align*}
{{\tilde V}_\varphi }(TK,L) &= \int_{{S^{n - 1}}} {\varphi \left( {\frac{{{\rho _{TK}}(u)}}{{{\rho _L}(u)}}} \right)d{{\tilde V}_{TK}}(u)}  \\
&= |T|\int_{{S^{n - 1}}} {\varphi \left( {\frac{{{\rho _K}\left( {\left\langle {{T^{ - 1}}u} \right\rangle } \right)}}{{{\rho _{{T^{ - 1}}L}}\left( {\left\langle {{T^{ - 1}}u} \right\rangle } \right)}}} \right)d{{\tilde V}_K}\left( {\left\langle {{T^{ - 1}}u} \right\rangle } \right)}  \\
&= |T|\int_{{S^{n - 1}}} {\varphi \left( {\frac{{{\rho _K}}}{{{\rho _{{T^{ - 1}}L}}}}} \right)d{{\tilde V}_K}}  \\
&= |T|{{\tilde V}_\varphi }(K,{T^{ - 1}}L),
\end{align*}
as desired.

From (1) and Definition \ref{def 3.1}, we have
\[{{\bar{\tilde V}}_\varphi }(TK,L) = {\varphi ^{ - 1}}\left( {\frac{{{{\tilde V}_\varphi }(TK,L)}}{{V(TK)}}} \right) = {\varphi
^{ - 1}}\left( {\frac{{{{\tilde V}_\varphi }(K,{T^{ - 1}}L)}}{{V(K)}}} \right) = {\bar{\tilde V}_\varphi }(K,{T^{ - 1}}L),\]
as desired.

Take $T=\lambda I_n$ in (2), it yields  (3) directly.
\end{proof}

Along with the functional $\tilde{V}_\varphi(K,L)$, we introduce
\begin{definition}\label{def 3.3}
Suppose $K, L\in\mathcal{S}^n_o$ and $\varphi\in\Phi$, define
\[{O_\varphi }(K,L) = {\left\| {\frac{{{\rho _K}}}{{{\rho _L}}}:{{\tilde V}_K}} \right\|_\varphi }= \inf \left\{ {\lambda  > 0:{\varphi ^{ - 1}}\left( {\int_{{S^{n - 1}}} {\varphi \left( {\frac{{{\rho _K}}}{{\lambda {\rho _L}}}} \right)dV_K^*} } \right) \le 1} \right\}.\]
\end{definition}

Obviously, $O_\varphi(K,K)=1$. If $\varphi(t)=t^p$, $1\le p<\infty$,
then $O_\varphi(K,L)={\bar{\tilde V}}_{-p}(K,L)$.

From Definition \ref{def 3.3} and Definition \ref{def 3.1},   we have
\begin{align*}
{O_\varphi }(K,L)
&= \inf \left\{ {\lambda  > 0:\frac{{{{\tilde V}_\varphi }( K, \lambda L)}}{{V(K)}} \le \varphi (1)} \right\} \\
&= \inf \left\{ {\lambda  > 0:{{\bar{\tilde V}}_\varphi }(K, \lambda L) \le 1} \right\}.
\end{align*}
Combining this with Lemma \ref{lem 3.2}, we immediately obtain
\begin{lemma}\label{lem 3.4}
Suppose $K,L\in\mathcal{S}^n_o$ and $\varphi\in\Phi$. Then

\noindent (1) $O_\varphi(TK,L)=O_\varphi(K,T^{-1}L)$, for all $T\in {\rm GL}(n)$.

\noindent (2) $O_\varphi(\lambda K, L)= O_\varphi(K,\lambda^{-1}L)=\lambda O_\varphi(K,L)$, for all $\lambda>0$.
\end{lemma}

The next lemma provides a simple but powerful identity.

\begin{lemma}\label{lem 3.5}
Suppose $K,L\in\mathcal{S}^n_o$ and $\varphi\in\Phi$. Then
\[{\bar{\tilde V}}_\varphi (K,{O_\varphi }(K,L)L) = 1.\]
Consequently, there is the following equivalence
\[{\bar{\tilde V}}_\varphi (K, L) = 1\quad \Longleftrightarrow\quad {O_\varphi }(K,L)=1. \]
\end{lemma}

\begin{proof}
From Definition \ref{def 3.1}, Definition \ref{def 3.3}, together
with Lemma \ref{lem 2.1}, it follows that
\[\varphi \left( {{{\bar{\tilde V}}_\varphi }(K,{O_\varphi }(K,L)L)} \right) = \int_{{S^{n - 1}}} {\varphi \left( {\frac{{{\rho _K}}}{{{O_\varphi }(K,L){\rho _K}}}} \right)dV_K^*}  = \varphi (1).\]
Thus, ${{\bar{\tilde V}}_\varphi }(K,{O_\varphi }(K,L)L) = 1$.  By
Lemma \ref{lem 2.1} again,  the desired equivalence follows.
\end{proof}

What follows establishes the dual Orlicz Minkowski inequalities.
\begin{lemma}\label{lem 3.6}
Suppose $K, L\in\mathcal{S}^n_o$ and $\varphi\in\Phi$. Then
\begin{equation}\label{ineq 3.6.1}
{{\bar{\tilde V}}_\varphi }(K,L) \ge {\left( {\frac{{V(K)}}{{V(L)}}} \right)^{\frac{1}{n}}},
\end{equation}
and
\begin{equation}\label{ineq 3.6.2}
{O_\varphi }(K,L) \ge {\left( {\frac{{V(K)}}{{V(L)}}} \right)^{\frac{1}{n}}}.
\end{equation}
Each equality holds in the above inequalities if and only if $K$ and $L$ are dilates.
\end{lemma}

\begin{proof}
From Definition \ref{def 3.1},  the fact that $\varphi^{-1}$ is strictly
increasing in $(0,\infty)$ together with the convexity of $\varphi$
and Jensen's inequality, the definition of $V_K^*$, and the reverse H$\rm{\ddot{o}}$lder inequality, we have
\begin{align*}
{{\bar{\tilde V}}_\varphi }(K,L) &= {\varphi ^{ - 1}}\left( {\int_{{S^{n - 1}}} {\varphi \left( {\frac{{{\rho _K}}}{{{\rho _L}}}} \right)dV_K^*} } \right) \\
&\ge {\varphi ^{ - 1}}\left( {\varphi \left( {\int_{{S^{n - 1}}} {\frac{{{\rho _K}}}{{{\rho _L}}}dV_K^*} } \right)} \right)\\
&= \frac{1}{{nV(K)}}\int_{{S^{n - 1}}} {\frac{{\rho _K^{n + 1}}}{{{\rho _L}}}dS}  \\
&\ge \frac{1}{{V(K)}}{\left( {\frac{1}{n}\int_{{S^{n - 1}}} {\rho _K^{(n + 1) \cdot \frac{n}{{n + 1}}}dS} } \right)^{\frac{{n + 1}}{n}}}{\left( {\frac{1}{n}\int_{{S^{n - 1}}} {\rho _L^{ - 1 \cdot ( - n)}dS} } \right)^{ - \frac{1}{n}}} \\ &= {\left( {\frac{{V(K)}}{{V(L)}}} \right)^{\frac{1}{n}}}.
\end{align*}

By the equality condition of the reverse H$\rm{\ddot{o}}$lder inequality, we know that
the equality in the forth line occurs only if $\rho_K/\rho_L$ is a positive constant on $S^{n-1}$.
Thus,  the equality holds in (\ref{ineq 3.6.1}) only if $K$ and $L$ are dilates.
Conversely, if $K=sL$ for some $s>0$, then
${{\bar{\tilde V}}_\varphi }(K,L)=s=\left(V(K)/V(L))\right)^{1/n}$.

Combining  Lemma \ref{lem 3.5} with inequality (\ref{ineq 3.6.1}), we have
\[1 = {\bar{\tilde V}}_\varphi  (K,{O_\varphi }(K,L)L) \ge {\left( {\frac{{V(K)}}{{V({O_\varphi }(K,L)L)}}} \right)^{\frac{1}{n}}}=\frac{1}{{O_\varphi }(K,L)}{\left( {\frac{{V(K)}}{{V(L)}}} \right)^{\frac{1}{n}}},\]
where the  equality holds if and only if $K$ and $O_\varphi(K,L)L$ are dilates.
Thus, inequality (\ref{ineq 3.6.2}), as well as its equality condition,
is derived.
\end{proof}

The next lemma is crucial to prove the continuity of the functionals ${\tilde V}_\varphi(K,L)$, ${\bar{\tilde V}}_\varphi(K,L)$
and $O_\varphi(K,L)$ in $(K,L,\varphi)$.
\begin{lemma}\label{lem 3.7}
Suppose $f_i, f$ are strictly positive and continuous functions on
$S^{n-1}$; $\varphi_k, \varphi\in\Phi$; $\mu_l,\mu$ are Borel
probability measures on $S^{n-1}$;  $i,k,l\in\mathbb N$. If $f_i\to f$ pointwise, $\varphi_k\to\varphi$,
and $\mu_l\to\mu$ weakly, then
\begin{equation}\label{crucial lim1}
\int_{{S^{n - 1}}} {{\varphi _k}\left( {{f_i}} \right)d{\mu _l}}  \to \int_{{S^{n - 1}}} {\varphi \left( f \right)d\mu },
\end{equation}

\begin{equation}\label{crucial lim2}
\varphi _k^{ - 1}\left( {\int_{{S^{n - 1}}} {{\varphi _k}\left( {{f_i}} \right)d{\mu _l}} } \right) \to {\varphi ^{ - 1}}\left( {\int_{{S^{n - 1}}}
{\varphi \left( f \right)d\mu }} \right),
\end{equation}
and

\begin{equation}\label{crucial lim3}
{\left\| {{f_i}:{\mu _l}} \right\|_{{\varphi _k}}} \to {\left\|
{f:\mu } \right\|_\varphi }.
\end{equation}
\end{lemma}

\begin{proof}
The continuity of $f_i$ and $f$,  and $f_i\to f$  pointwise
guarantee that $f_i\to f$ uniformly. Thus,
there exists an $N_0\in\mathbb N$, such that
\[\frac{1}{2}\mathop {\min }_{u \in {S^{n - 1}}} f(u) \le {f_i} \le 2\mathop {\max }_{u \in {S^{n - 1}}} f(u),\quad {\rm for}\; i>N_0.\]

Let
\[{c_m} = \min \left\{ {\frac{1}{2}\mathop {\min }_{u \in {S^{n - 1}}} f(u),\mathop {\min }_{u \in {S^{n - 1}}} {f_i}(u),\;{\rm{with}}\;i \le N_0} \right\},\]
and
\[{c_M} = \max \left\{ {2\mathop {\max }_{u \in {S^{n - 1}}} f(u),\mathop {\max }_{u \in {S^{n - 1}}} {f_i}(u),\;{\rm{with}}\;i \le N_0} \right\}.\]
The strictly positivity and the continuity of $f_i$ and $f$ imply that
\[0<c_m\le c_M<\infty.\]

Thus,
\begin{equation}\label{(3.6)}
c_m\le f(u)\le c_M\quad {\rm and}\quad c_m\le f_i(u)\le c_M,   \quad {\rm for}\; u\in S^{n-1}\; {\rm and}\; i\in\mathbb N.
\end{equation}

Since $\varphi_k\to\varphi$ uniformly on $[c_m,c_M]$,   by (\ref{(3.6)}) and  that $f_i\to f$ uniformly, it follows that as $i,k\to\infty$,
\[\varphi_k(f_i)\to \varphi(f),\quad  {\rm uniformly\; on}\; S^{n-1}.\]
Combined with that $\mu_l\to\mu$ weakly, it concludes that as $i,k,l\to\infty$,
\[\int_{{S^{n - 1}}} {{\varphi _k}\left( {{f_i}} \right)d{\mu _l}}  \to \int_{{S^{n - 1}}} {\varphi \left( f \right)d\mu }, \]
as (\ref{crucial lim1}) desired.

Now, we proceed to prove (\ref{crucial lim2}).

For brevity, let
\[{a_{i,k,l}} = \int_{{S^{n - 1}}} {{\varphi _k}\left( {{f_i}} \right)d{\mu _l}} \quad{\rm{and}}\quad a = \int_{{S^{n - 1}}} {\varphi \left( f \right)d\mu } .\]
Then
\[\varphi ({c_m}) \le a \le \varphi ({c_M})\quad {\rm{and}}\quad {\varphi _k}({c_m}) \le {a_{i,k,l}} \le {\varphi _k}({c_M}), \quad {\rm for}\; i,k,l\in\mathbb N.\]

Let
\[{a_m} = \inf \left\{ {\varphi ({c_m}),{\varphi _k}({c_m}),\;{\rm{with}}\;k \in \mathbb N} \right\},\]
and
\[{a_M} = \sup \left\{ {\varphi ({c_M}),{\varphi _k}({c_M}),\;{\rm{with}}\;k \in \mathbb N} \right\}.\]
Since $\varphi_k(c_m)\to\varphi(c_m)$ and $\varphi_k(c_M)\to\varphi(c_M),$  it gives
\[ 0<a_m\le a_M<\infty\quad {\rm and}\quad a,a_{i,k,l}\in [a_m,a_M], \quad {\rm for}\; i,k,l\in\mathbb N. \]

Since $\varphi_k\to\varphi$ uniformly on $[c_m,c_M]$, it follows that
\[ \varphi^{-1}_k\to\varphi^{-1},\quad {\rm  uniformly\; on}\; [a_m,a_M].\]
Thus, from that $a_{i,k,l}\to a$ as $i,k,l\to\infty$, it follows that
\[ \varphi^{-1}_k(a_{i,k,l})\to\varphi^{-1}(a),\quad {\rm as}\; i,k,l\to\infty,\]
as (\ref{crucial lim2}) desired.

Finally, we conclude to  show (\ref{crucial lim3}).

At first, we prove that the  set
$\left\{\|f_i:\mu_l\|_{\varphi_k}: i,k,l\in\mathbb N\right\} $  is bounded.

Indeed, from  (\ref{(3.6)}) together with the strict monotonicity of $\varphi$ and $\varphi^{-1}$, Lemma \ref{lem 2.1},
and (\ref{(3.6)}) together with the strict monotonicity of $\varphi$ and $\varphi^{-1}$ again, it follows that
\begin{align*}
\frac{{{c_m}}}{{{{\left\| {{f_i}:{\mu _l}} \right\|}_{{\varphi _k}}}}} &\le \varphi _k^{ - 1}\left( {\int_{{S^{n - 1}}} {{\varphi _k}\left( {\frac{{{f_i}}}{{{{\left\| {{f_i}:{\mu _l}} \right\|}_{{\varphi _k}}}}}} \right)d{\mu _l}} } \right)\\&= 1 \\
&\le  \frac{{{c_M}}}{{{{\left\| {{f_i}:{\mu _l}} \right\|}_{{\varphi _k}}}}},
\end{align*}
which immediately gives
\[{c_m} \le {\left\| {{f_i}:{\mu _l}} \right\|_{{\varphi _k}}} \le {c_M},\quad{\rm{for}}\;i,k,l \in \mathbb N.\]

Now, we can complete the proof of  (\ref{crucial lim3}).

Since $\left\{\|f_i:\mu_l\|_{\varphi_k}: i,k,l\in\mathbb N\right\} $
is bounded, to prove (\ref{crucial lim3}), it suffices to prove that
each convergent subsequence
$\{\|f_{i_p}:\mu_{l_r}\|_{\varphi_{k_q}}\}_{p,q,r\in\mathbb N}$ of $\left\{\|f_i:\mu_l\|_{\varphi_k}: i,k,l\in\mathbb N\right\} $
necessarily converges to $\|f:\mu\|_\varphi$, as $i_p,k_q,l_r\to\infty$.

Assume
\[\mathop {\lim }\limits_{p,q,r \to \infty } {\left\| {{f_{{i_p}}}:{\mu _{{l_r}}}} \right\|_{{\varphi _{{k_q}}}}} = {\lambda _0}.\]

Note that
\[\frac{{{f_{{i_p}}}}}{{{{\left\| {{f_{{i_p}}}:{\mu _{{l_r}}}} \right\|}_{{\varphi _{{k_q}}}}}}} \to \frac{f}{{{\lambda _0}}}\; {\rm pointwise},\quad{\varphi _{{k_q}}} \to \varphi ,\quad{\rm{and}}\quad{\mu _{{l_r}}} \to \mu \;{\rm{weakly}},\]
by (\ref{crucial lim2}), we have
\[\mathop {\lim }_{p,q,r \to \infty } \varphi _{{k_q}}^{ - 1}\left( {\int_{{S^{n - 1}}} {{\varphi _{{k_q}}}\left( {\frac{{{f_{{i_p}}}}}{{{{\left\| {{f_{{i_p}}}:{\mu _{{l_r}}}} \right\|}_{{\varphi _{{k_q}}}}}}}} \right)d{\mu _{{l_r}}}} } \right) = {\varphi ^{ - 1}}\left( {\int_{{S^{n - 1}}} {\varphi \left( {\frac{f}{{{\lambda _0}}}} \right)d\mu } } \right).\]

Meanwhile, since
\[\varphi _{{k_q}}^{ - 1}\left( {\int_{{S^{n - 1}}} {{\varphi _{{k_q}}}\left( {\frac{{{f_{{i_p}}}}}{{{{\left\| {{f_{{i_p}}}:{\mu _{{l_r}}}} \right\|}_{{\varphi _{{k_q}}}}}}}} \right)d{\mu _{{l_r}}}} } \right) = 1, \quad {\rm for\; each}\; (p,q,r),\]
it yields that
\[\mathop {\lim }\limits_{p,q,r \to \infty } \varphi _{{k_q}}^{ - 1}\left( {\int_{{S^{n - 1}}} {{\varphi _{{k_q}}}\left( {\frac{{{f_{{i_p}}}}}{{{{\left\| {{f_{{i_p}}}:{\mu _{{l_r}}}} \right\|}_{{\varphi _{{k_q}}}}}}}} \right)d{\mu _{{l_r}}}} } \right) = 1.\]

Hence,
\[{\varphi ^{ - 1}}\left( {\int_{{S^{n - 1}}} {\varphi \left( {\frac{f}{{{\lambda _0}}}} \right)d\mu } } \right) = 1.\]
From Lemma \ref{lem 2.1}, it follows that
$\lambda_0=\|f:\mu\|_\varphi$.

The proof is complete.
\end{proof}

Using Lemma \ref{lem 3.7}, we  immediately obtain
\begin{lemma}\label{lem 3.8}
Suppose $K,K_i,L,L_j\in\mathcal{S}^n_o$ and
$\varphi,\varphi_k\in\Phi$,  $i,j,k\in\mathbb N$. If $K_i\to K$, $L_j\to L$ and $\varphi_k\to\varphi$, then
\[\mathop {\lim }\limits_{i,j,k \to \infty } {{\tilde V}_{{\varphi _k}}}({K_i},{L_j}) = {{\tilde V}_\varphi }(K,L),\]
\[\mathop {\lim }\limits_{i,j,k \to \infty } {{\bar{\tilde V}}_{{\varphi _k}}}({K_i},{L_j}) = {{\bar{\tilde V}}_\varphi }(K,L),\]
and
\[\mathop {\lim }\limits_{i,j,k \to \infty } {O_{{\varphi _k}}}({K_i},{L_j}) = {O_\varphi }(K,L).\]
\end{lemma}
\begin{proof}
That $K_i\to K$ and $L_j\to L$  yields
$\rho_{K_i}/\rho_{L_j}$ and $\rho_K/\rho_L$ are strictly positive
continuous on  $S^{n-1}$; $\rho_{K_i}/\rho_{L_j}\to \rho_{K}/\rho_L$; $\tilde{V}_{K_i}\to \tilde{V}_{K}$ weakly, and
$V_{K_i}^*\to V_K^*$ weakly.  Combining these facts  and applying Lemma \ref{lem 3.7},   the desired
limits can be derived directly.
\end{proof}

Recall that
\[{\bar{\tilde V}}_{-1}(K,L)=\int_{S^{n-1}}{\frac{\rho_K}{\rho_L} dV_K^*},\quad {\rm for}\; K,L\in\mathcal{S}^n_o. \]

The next lemma will be used in Section 6.
\begin{lemma}\label{lem 3.9}
Suppose $K,L\in\mathcal{S}^n_o$, $\varphi\in\Phi$ and $p\in [1,\infty)$. Then

\noindent (1) $ {\bar{\tilde V}}_{\varphi^p}(K,L) $ is increasing and bounded from above in $p$,
and bounded from below by ${\bar{\tilde V}}_{-1}(K,L)$.

\noindent (2) $\lim\limits_{p\to\infty}{\bar{\tilde V}}_{\varphi^p}(K,L) =\left\|\frac{\rho_K}{\rho_L} \right\|_\infty$.

\noindent (3) $O_{\varphi^p}(K,L)$ is increasing and bounded from above in $p$, and
bounded from below by ${\bar{\tilde V}}_{-1}(K,L)$.

\noindent (4) $\lim\limits_{p\to\infty} O_{\varphi^p}(K,L)=\left\|\frac{\rho_K}{\rho_L} \right\|_\infty$.
\end{lemma}

\begin{proof}
Let $\lambda\in (0,\infty)$. From Definition \ref{def 3.1}, we have
\[{{\bar{\tilde V}}_{{\varphi ^p}}}(K,\lambda L) = {\varphi ^{ - 1}}\left( {{{\left( {\int_{{S^{n - 1}}} {\varphi {{\left( {\frac{{{\rho _K}}}{{\lambda {\rho _L}}}} \right)}^p}dV_K^*} } \right)}^{1/p}}} \right).\]
By Jensen's inequality,  ${{{\left( {\int_{{S^{n - 1}}} {\varphi {{\left( {\frac{{{\rho _K}}}{{\lambda {\rho _L}}}} \right)}^p}dV_K^*} } \right)}^{1/p}}}$
is increasing in $p\in [1,\infty)$.  Since $\varphi^{-1}$ is also increasing in
$(0,\infty)$, it yields that ${{\bar{\tilde V}}_{{\varphi ^p}}}(K,\lambda L)$
is increasing in $p\in [1,\infty)$.

Since $\varphi^{-1}$ and  $\varphi$ are both continuous and strictly
increasing on $[0,\infty)$, it follows that
\begin{align*}
\mathop {\lim }\limits_{p \to \infty } {{\bar{\tilde V}}_{{\varphi ^p}}}(K,\lambda L) &= \mathop {\lim }_{p \to \infty } {\varphi ^{ - 1}}\left( {{{\left( {\int_{{S^{n - 1}}} {\varphi {{\left( {\frac{{{\rho _K}}}{{\lambda {\rho _L}}}} \right)}^p}dV_K^*} } \right)}^{1/p}}} \right) \\
&= {\varphi ^{ - 1}}\left( {\mathop {\lim }_{p \to \infty } {{\left( {\int_{{S^{n - 1}}} {\varphi {{\left( {\frac{{{\rho _K}}}{{\lambda {\rho _L}}}} \right)}^p}dV_K^*} } \right)}^{1/p}}} \right) \\
&= {\varphi ^{ - 1}}\left( {\max \left\{ {\varphi \left( {\frac{{{\rho _K}(u)}}{{\lambda {\rho _L}(u)}}} \right):u \in {S^{n - 1}}} \right\}} \right) \\
&= {\varphi ^{ - 1}}\left( {\varphi \left( {\max \left\{ {\frac{{{\rho _K}(u)}}{{\lambda {\rho _L}(u)}}:u \in {S^{n - 1}}} \right\}} \right)} \right) \\
&= \frac{1}{\lambda }{\left\| {\frac{{{\rho _K}}}{{{\rho _L}}}} \right\|_\infty}.
\end{align*}
Thus, $ {\bar{\tilde V}}_{\varphi^p}(K,\lambda L) $ is bounded from above by
$\frac{1}{\lambda }{\left\| {\frac{{{\rho _K}}}{{{\rho _L}}}} \right\|_\infty }$.

From the definition of ${\bar{\tilde V}}_{-1}(K,\lambda L)$, the
strict monotonicity of $\varphi^{-1}$ together with the convexity of
$\varphi$ and Jensen's inequality, and  the definition of ${\bar{\tilde V}}_\varphi(K,L)$, we have
\begin{align*}
{\bar{\tilde V}}_{-1}(K,\lambda L) &= {\varphi ^{ - 1}}\left( {\varphi \left( {\int_{{S^{n - 1}}} {\frac{{{\rho _K}}}{{\lambda {\rho _L}}}dV_K^*} } \right)} \right) \\
&\le {\varphi ^{ - 1}}\left( {\int_{{S^{n - 1}}} {\varphi \left( {\frac{{{\rho _K}}}{{\lambda {\rho _L}}}} \right)dV_K^*} } \right)\\
&={\bar{\tilde V}}_\varphi(K,\lambda L).
\end{align*}
Thus, ${\bar{\tilde V}}_{-1}(K,\lambda L) \le {\bar{\tilde V}}_\varphi(K,\lambda L)$.

Let $\lambda=1$, it gives (1) and (2) directly.

Recall that
\[{O_{{\varphi ^p}}}(K,L) = \inf \left\{ {\lambda  > 0:{\bar{\tilde V}}_{\varphi^p}(K,\lambda^{-1} L) \le 1} \right\}.\]
So, for $1\le p<q<\infty$, from (1) we have
\[{\bar{\tilde V}}_{-1}(K,\lambda^{-1} L) \le {{\bar{\tilde V}}_\varphi }(K,\lambda^{-1} L)
\le {{\bar{\tilde V}}_{{\varphi ^p}}}(K, \lambda^{-1} L) \le
{{\bar{\tilde V}}_{{\varphi ^q}}}(K,\lambda^{-1} L) \le \lambda {\left\| {\frac{{{\rho _K}}}{{{\rho _L}}}} \right\|_\infty }.\]
Thus, we obtain
\[{\bar{\tilde V}}_{-1}(K, L) \le {O_\varphi }(K, L) \le {O_{{\varphi^p}}}(K, L) \le {O_{{\varphi ^q}}}(K, L) \le {\left\| {\frac{{{\rho _K}}}{{{\rho _L}}}} \right\|_\infty }, \]
which implies (3) immediately.

 By (3), any  subsequence $\left\{O_{\varphi^{p_j}}(K,L)\right\}_{j}$, with
$\lim\limits_{j\to\infty}p_j=\infty$,  must converge to certain number
$\lambda_0\in [{\bar{\tilde V}}_{-1}(K, L), \|\rho_K/\rho_L\|_\infty]$. So, to prove (4), it suffices to prove
\[\lambda_0 = {\left\| {\frac{{{\rho _K}}}{{{\rho _L}}}} \right\|_\infty }.\]

For brevity, let
\[\lambda_\infty={\left\| {\frac{{{\rho _K}}}{{{\rho _L}}}} \right\|_\infty }\quad {\rm and} \quad {\lambda _j} = O_{\varphi^{p_j}}(K,L).\]

For each $j$, define
\[{g_j}(\lambda ) = {\left[ { \int_{{S^{n - 1}}} {\varphi {{\left( {\frac{{{\rho _K}}}{{\lambda {\rho _L}}}} \right)}^{{p_j}}}d{V^*_K}} } \right]^{1/{p_j}}},\]
and
\[{g_\infty }(\lambda ) = \varphi\left({\lambda ^{ - 1}}{\left\| {\frac{{{\rho _K}}}{{{\rho _L}}}} \right\|_\infty }\right).\]
Note that the functions $g_j$ and $g_\infty$ are continuous on $[\lambda_1,\lambda_\infty]$,  and
$g_j\to g_\infty$ pointwise on $[\lambda_1,\lambda_\infty]$ by (1). Thus,
$g_j\to g_\infty$, uniformly on $[\lambda_1,\lambda_\infty]$.

Consequently, we have
\[\mathop {\lim }_{j \to \infty } {g_j}({\lambda _j}) = \left( {\mathop {\lim }_{j \to \infty } {g_j}} \right)(\mathop {\lim }_{j \to \infty } {\lambda _j}) = {g_\infty }({\lambda _0}).\]
Note that $g_j(\lambda_j)=\varphi(1)$ for each $j$. Hence, we obtain
\[ {g_\infty }({\lambda _0})=\varphi (1) ;\quad {\rm i.e.},\quad {\lambda _0} = {\left\| {\frac{{{\rho _K}}}{{{\rho _L}}}} \right\|_\infty }.\]

The proof is complete.
\end{proof}

\vskip 25pt
\section{\bf Orlicz-Legendre ellipsoids}
\vskip 10pt

Let $K\in \mathcal{S}^n_o$ and $\varphi\in\Phi$. For any $T\in {\rm SL}(n)$, by Lemma \ref{lem 3.6} it gives
\[{{\bar{\tilde V}}_\varphi }(K,TB) \ge {\left( {\frac{{V(K)}}{{{\omega _n}}}} \right)^{\frac{1}{n}}}\quad {\rm and}\quad O_\varphi(K,TB)\ge  {\left( {\frac{{V(K)}}{{{\omega _n}}}} \right)^{\frac{1}{n}}}.\]

In view of the intimate connection between $\bar{\tilde V}_\varphi$  and $O_\varphi$, to find the so-called Orlicz-Legendre ellipsoids,
we also consider the following three problems, which are closely related to our originally posed Problem ${\tilde S}_\varphi$.

\noindent\textbf{Problem ${\rm P}_1$.}  \emph{Find an ellipsoid $E$, amongst all origin-symmetric
ellipsoids, which solves the  constrained minimization problem}
\[\min  {\bar{\tilde V}}_\varphi (K,E)\quad{\rm{subject}}\;{\rm{to}}\quad V(E) \le {\omega _n}.\]

\noindent\textbf{Problem ${\rm P}_2$.} \emph{Find an ellipsoid $E$, amongst all origin-symmetric
ellipsoids, which solves the  constrained minimization problem}
\[\min O_\varphi (K,E)\quad{\rm{subject}}\;{\rm{to}}\quad V(E) \le {\omega _n}.\]

The homogeneity of volume functional and Orlicz norm prompts us to
consider the following Problem ${\rm P}_3$, which is in some sense  dual to Problem ${\rm P}_2$.

\noindent\textbf{Problem ${\rm P}_3$.} \emph{Find an ellipsoid $E$, amongst all origin-symmetric ellipsoids,
which solves the  constrained maximization problem}
\[\max \left(\frac{{{\omega _n}}}{{V(E)}}\right)^{\frac{1}{n}}\quad{\rm{subject}}\;{\rm{to}}\quad O_\varphi (K,E)\le 1.\]

In order to convenient comparison, we restate Problem ${\tilde S}_\varphi$ as the following.

\noindent\textbf{Problem ${\tilde S}_\varphi$.} \emph{Find an
ellipsoid $E$, amongst all origin-symmetric ellipsoids, which solves
the constrained maximization problem}
\[\max \left(\frac{{{\omega _n}}}{{V(E)}}\right)^{\frac{1}{n}}\quad{\rm{subject}}\;{\rm{to}}\quad {\bar{\tilde V}}_\varphi (K,E)\le 1.\]

Two observations are in order. First, from Definition \ref{def 3.1} together with the fact that $\varphi^{-1}$ is strictly increasing in
$(0,\infty)$, the objective functional in ${\rm P}_1$ can be
replaced by $\tilde V_\varphi(K,E)$. Second, by the fact
$V(E)V(E^*)=\omega_n^2$, the  objective functional in ${\rm P}_3$ and
${\tilde S}_\varphi$ can be replaced by $V(E^*)$.

This section is organized as follows.  After proving Lemmas \ref{lem 4.1} and \ref{lem 4.2}, 
we prove Theorems \ref{thm 4.3} and \ref{thm 4.4}, which demonstrate the existence and  uniqueness of solution to
${\rm P}_1$, respectively. The connection between ${\rm P}_1$ and
${\rm P}_2$ is established by Lemma \ref{lem 4.5}, then the unique
existence of solution to  ${\rm P}_2$ is shown in  Theorem 
\ref{thm 4.6}.  Theorem \ref{thm 4.7} shows that the solutions to 
${\rm P}_2$ and ${\rm P}_3$ only differ by a scale factor. Thus,  the
unique existence  of solution to ${\rm P}_3$ is confirmed.  Lemma
\ref{lem 4.8} reveals that  ${\rm P}_3$ and ${\tilde S}_\varphi$ are
essentially identical, so the proof of the unique existence  of
solution to ${\tilde S}_\varphi$ is complete. Therefore, the notion
of  Orlicz-Legendre ellipsoid is ready to come out.

\begin{lemma}\label{lem 4.1}
Suppose $K\in\mathcal{S}^n_o$ and $\varphi\in\Phi$. Then
\[\mathop {\lim }\limits_{\scriptstyle T \in {\rm SL}(n) \atop \scriptstyle \left\| T \right\| \to \infty } {{\tilde V}_\varphi }(TK,B) = \infty, \]
and
\[ \mathop {\lim }\limits_{\scriptstyle T \in {\rm{SL}}(n) \atop \scriptstyle \left\| T \right\| \to \infty } O_\varphi(TK,B)= \infty .\]
\end{lemma}

\begin{proof}
Let $r_K=\min\limits_{S^{n-1}}\rho_K$. Then $r_KB\subseteq K$. In
addition, there exists a positive  $r>0$, say
$r=\frac{1}{\sqrt{n}}r_K$, such that $r[-1,1]^n\subseteq r_KB$. For
$T\in {\rm SL}(n)$, write $T$ in the form $ T=O_1AO_2$, where $A$ is
an $n\times n$ diagonal matrix, with $\det(A)=1$ and positive
diagonal elements $a_1,\cdots, a_n$, and $O_1, O_2$ are $n\times n$
orthogonal matrices.

From the definition of the measure $ {{\tilde V}_{A{O_2}K}} $, the
polar coordinate formula, and the fact $K\supseteq r_K B$, the
orthogonality of $O_2$, the fact
$r_KB\supseteq r[-1,1]^n$, and finally the symmetry of $Ax$ in $x$
and $[-1,1]^n$ with respect to $o$, we have
\begin{align*}
\int_{{S^{n - 1}}} {{\rho _{A{O_2}K}}d{{\tilde V}_{A{O_2}K}}}  &= \frac{1}{n}\int_{{S^{n - 1}}} {\rho _{A{O_2}K}^{n + 1}dS} \\ &= \frac{{n + 1}}{n}\int_K {|A{O_2}x|dx}
\\&\ge \frac{{n + 1}}{n}\int_{{r_K}B} {|A{O_2}x|dx} \\ &= \frac{{n + 1}}{n}\int_{{r_K}B} {|Ax|dx}  \\
&\ge \frac{{n + 1}}{n}\int_{r{{[ - 1,1]}^n}} {|Ax|dx} \\ &= \frac{{(n + 1)2^n{r^{n + 1}}}}{n}\int_{{{[0,1]}^n}} {|Ax|dx}.
\end{align*}

For any $y\in\mathbb R^n$, let $\|y\|_1$ denote the $l_1$ norm of
$y$. Recall that there exists a positive  $C$ such that $|y|\ge C\|y\|_1$, and
$ \sum\limits_{i=1}^n a_i\ge \max\limits_{1\le i\le n}a_i=\|A\|=\|T\|$. So, we have
\[\int_{{{[0,1]}^n}} {|Ax|dx}  \ge \int_{{{[0,1]}^n}} {C{{\left\| {Ax} \right\|}_1}dx}  \ge \frac{C}{2}\sum\limits_{i = 1}^n {{a_i}}  \ge \frac{C}{2}\left\| T \right\|.\]

Thus, we obtain
\begin{equation}\label{existence ineq}
\int_{{S^{n - 1}}} {{\rho _{A{O_2}K}}d{{\tilde V}_{A{O_2}K}}} \ge \frac{{(n + 1){2^{n - 1}}{r^{n + 1}}C}}{n}\left\| T \right\|.
\end{equation}

Now, from Definition \ref{def 3.1} together with Lemma \ref{lem 3.2} (1), the convexity of $\varphi$
together with Jensen's inequality, the strict
monotonicity of $\varphi$  together with (\ref{existence ineq}), and
the fact $V(AO_2 K)=V(K)$, we have
\begin{align*}
\frac{{{{\tilde V}_\varphi }(TK,B)}}{{V(TK)}}
&= \frac{1}{{V(A{O_2}K)}}\int_{{S^{n - 1}}} {\varphi \left( {{\rho _{A{O_2}K}}} \right)d{{\tilde V}_{A{O_2}K}}}\\
&\ge \varphi \left( {\frac{1}{{V(AO_2K)}}\int_{{S^{n - 1}}} {{\rho _{A{O_2}K}}d{{\tilde V}_{A{O_2}K}}} } \right)\\
&\ge \varphi \left( \frac{{(n + 1){2^{n - 1}}{r^{n + 1}}C}}{nV(K)}\left\| T \right\|  \right).
\end{align*}
That is,
\begin{equation}\label{crucial ineq}
\frac{{{{\tilde V}_\varphi }(TK,B)}}{{V(K)}} \ge \varphi \left( {\frac{{{2^{n-1}}{r^{n + 1}}(n + 1)C}}{{nV(K)}}\left\| T \right\|} \right).
\end{equation}
By the strict monotonicity of $\varphi$ again, it immediately yields
\[\mathop {\lim }\limits_{\scriptstyle T \in {\rm SL}(n) \atop \scriptstyle \left\| T \right\| \to \infty } {{\tilde V}_\varphi }(TK,B) = \infty .\]

Let
\[{C_1} = \frac{{{2^{n-1}}{r^{n + 1}}(n + 1)C}}{{nV(K)}}.\]

Note that $r$ depends on $K$.  Applying (\ref{crucial ineq}) to the star body $O_\varphi(TK,B)^{-1}K$ and using Lemma \ref{lem 3.5}, we obtain
\[\varphi \left( 1 \right) = \frac{{{{\tilde V}_\varphi }({O_\varphi }{{(TK,B)}^{ - 1}}TK,B)}}{{V\left( {{O_\varphi }{{(TK,B)}^{ - 1}}K} \right)}} \ge \varphi \left( {{O_\varphi }{{(TK,B)}^{ - 1}}{C_1}\left\| T \right\|} \right).\]
So, from the injectivity of $\varphi$, it follows that
\[O_\varphi(TK,B) \ge {C_1}\left\| T \right\|,\]
which   immediately yields
\[\mathop {\lim }\limits_{\scriptstyle T \in {\rm{SL}}(n) \atop \scriptstyle \left\| T \right\| \to \infty } {O_\varphi }(TK,B) = \infty ,\]
as desired.
\end{proof}

From Lemmas \ref{lem 4.1}, \ref{lem 2.2}, \ref{lem 3.2} and \ref{lem 3.4}, we immediately obtain
\begin{lemma}\label{lem 4.2}
Suppose $T\in {\rm SL}(n)$. Then
\[\mathop {\lim }\limits_{\scriptstyle T \in {\rm{SL}}(n)  \atop \scriptstyle \left\| T \right\| \to \infty } {{\tilde V}_\varphi }(K,TB) = \infty,\]
and
\[\mathop {\lim }\limits_{\scriptstyle T \in {\rm{SL}}(n) \atop \scriptstyle \left\| T \right\| \to \infty } O_\varphi(K,TB) = \infty .\]
\end{lemma}

Now, using Lemmas \ref{lem 4.2} and \ref{lem 2.3}, we can prove the existence of solution to problem ${\rm P}_1$.

\begin{theorem}\label{thm 4.3}
Suppose $K\in\mathcal{S}^n_o$ and $\varphi\in\Phi$. Then there
exists an solution to ${\rm P}_1$.
\end{theorem}

\begin{proof}
First, we prove that any $E\in\mathcal{E}^n$ with $V(E)<\omega_n$
cannot be a solution to  ${\rm P}_1$.

Indeed, let $\lambda_0=\left(\omega_n/V(E)\right)^{1/n}$, then $\lambda_0 E$
also satisfies the constraint condition in  ${\rm P}_1$. From
the fact that $\varphi$ is strictly increasing on $[0,\infty)$
together with Definition \ref{def 3.1}, it  necessarily results in that
${\tilde V}_\varphi(K,\lambda_0 E)<{\tilde V}_\varphi(K,E)$.

Hence, Problem ${\rm P}_1$ can be equivalently restated as
\[\inf \left\{ {{{\tilde V}_\varphi }(K,TB):T \in {\rm{SL}}(n)} \right\}.\]

Observe that the infimum exists, since
\[V(K)\varphi \left( {{{\left( {\frac{{V(K)}}{{{\omega _n}}}} \right)}^{\frac{1}{n}}}} \right) \le \inf \left\{ {{{\tilde V}_\varphi }(K,TB):T \in {\rm{SL}}(n)} \right\} \le {\tilde V_\varphi }(K,B) < \infty, \]
where the left inequality follows from Lemma \ref{lem 3.6} and
Definition \ref{def 3.1}.

Let
\[{\mathcal T} = \left\{ {T \in {\rm{SL}}(n):{{\tilde V}_\varphi }(K,TB) \le {{\tilde V}_\varphi }(K,B)} \right\}.\]
From Lemma \ref{lem 2.2} (3) and Lemma \ref{lem 3.8}, ${\tilde V}_\varphi(K,TB)$ is continuous in $T\in ({\rm SL}(n),d_n)$. Thus,
the set ${\mathcal T}$ is closed in $({\rm SL}(n),d_n)$. Meanwhile,
the definition of ${\mathcal T}$ and Lemma \ref{lem 4.2} guarantee
that ${\mathcal T}$ is bounded in $ ({\rm SL}(n),d_n)$. Hence,
$\mathcal T$ is compact.

Now, since ${\tilde V}_\varphi(K,TB)$ is continuous on $({\mathcal T},d_n)$, it
concludes that there exists a $T_0\in {\mathcal T}$ such that
\[{\tilde V}_\varphi (K,T_0B)=\min\{{\tilde V}_\varphi (K,TB): T\in {\mathcal T}\}= \inf\{{\tilde V}_\varphi (K,TB): T\in {\rm SL}(n)\},\]
which completes the proof.
\end{proof}

\begin{theorem}\label{thm 4.4}
Suppose $K\in\mathcal{S}^n_o$ and $\varphi\in\Phi$. Then, modulo
orthogonal transformations, there exists a unique ${\rm SL}(n)$
transformation solving  the extremal problem
\[\min \left\{ {{{\tilde V}_\varphi }(K,TB):T \in {\rm{SL}}(n)} \right\}.\]
Equivalently, there exists a unique solution to Problem ${\rm P}_1$.
\end{theorem}

\begin{proof}
The existence is shown by Theorem \ref{thm 4.3}.  We only need to prove the
uniqueness. For this aim, we argue by contradiction.

Assume that $T_1, T_2\in {\rm SL}(n)$ both solve the considered
minimization problem. Let $E_1=T_1B$, $E_2=T_2B$. It is known that each $T\in {\rm SL}(n)$ can be represented
in the form $T=PQ$, where $P$ is symmetric, positive definite and
$Q$ is orthogonal. So, w.l.o.g., we may assume
that $T_1, T_2$ are symmetric and positive definite.

By the Minkowski inequality for symmetric and positive definite matrices, we have
\[\det\left( \frac{T_1^{-1}+T_2^{-1}}{2} \right)^{\frac{1}{n}} > \frac{1}{2}\det (T_1^{-1})^{\frac{1}{n}} + \frac{1}{2}\det (T_2^{-1})^{\frac{1}{n}}=1. \]

Let
\[ T_3^{-1}= \det\left( \frac{T_1^{-1}+T_2^{-1}}{2} \right)^{-\frac{1}{n}} \frac{T_1^{-1}+T_2^{-1}}{2}.\]
Then $T_3\in {\rm SL}(n)$ is symmetric.

Let $E_3=T_3B$.  For all $u\in S^{n-1}$, we have
\begin{align*}
& h_{E_3^*}(u) =h_{T_3^{-1}B}(u)\\
& < h_{\frac{T_1^{-1}+T_2^{-1}}{2} B}(u)= \left|\frac{T_1^{-1}u+T_2^{-1}u}{2}\right|\\
& \le \frac{|T_1^{-1}u|+|T_2^{-1}u|}{2}= \frac{1}{2} h_{T_1^{-1}B} + \frac{1}{2} h_{T_1^{-1}B}.
\end{align*}

Since $E_i^*=T_i^{-1}B$, $i=1,2,3$, it follows that
\[{\tilde V}_\varphi(K,E_i)=\int_{S^{n-1}}{ \varphi\left( \rho_K h_{T_i^{-1}B}\right) d{\tilde V}_K  }.\]
From the fact that  $\varphi$ is strictly increasing and convex in $[0,\infty)$, we have
\[\varphi\left( \rho_K h_{T_3^{-1}B} \right) < \frac{1}{2} \varphi\left( \rho_K h_{T_1^{-1}B} \right) + \frac{1}{2} \varphi\left( \rho_K h_{T_2^{-1}B} \right). \]
Thus,
\[{\tilde V}_\varphi(K,E_3)<\frac{1}{2}{\tilde V}_\varphi(K,E_1)+\frac{1}{2}{\tilde V}_\varphi(K,E_2). \]
Hence,
\[ {\tilde V}_\varphi(K,E_3)< {\tilde V}_\varphi(K,E_1)={\tilde V}_\varphi(K,E_2).\]

However, from $T_3\in {\rm SL}(n)$ and the assumption on $E_1$ and $E_2$,  we also have
\[ {\tilde V}_\varphi(K,E_3) \ge {\tilde V}_\varphi(K,E_1)={\tilde V}_\varphi(K,E_2),\]
which contradicts the above. This completes the proof.
\end{proof}

\begin{lemma}\label{lem 4.5}
Suppose $E_0\in\mathcal{E}^n$ and $V(E_0)=\omega_n$. Then, for any
$T\in {\rm SL}(n)$,
\[{{\tilde V}_\varphi }\left( {K,{O_\varphi }(K,{E_0}){E_0}} \right) \le {{\tilde V}_\varphi }\left( {K,{O_\varphi }(K,{E_0})T{E_0}} \right)\]
if and only if
\[{O_\varphi }(K,{E_0}) \le {O_\varphi }(K,T{E_0}).\]
\end{lemma}

\begin{proof}
From Definition \ref{def 3.1} together with the strict monotonicity of $\varphi^{-1}$, Lemma \ref{lem 3.5},
and  Lemma \ref{lem 2.1} together with Definition \ref{def 3.3}, it follows that
\begin{align*}
&{{\tilde V}_\varphi }\left( {K,{O_\varphi }(K,{E_0}){E_0}} \right) \le {{\tilde V}_\varphi }\left( {K,{O_\varphi }(K,{E_0})T{E_0}} \right) \\
&\quad\quad \Longleftrightarrow\quad {{\bar{\tilde V}}_\varphi }\left( {K,{O_\varphi }(K,{E_0}){E_0}} \right) \le {{\bar{\tilde V}}_\varphi }\left( {K,{O_\varphi }(K,{E_0})T{E_0}} \right)\\
&\quad\quad \Longleftrightarrow\quad 1 \le {{\bar{\tilde V}}_\varphi }\left( {K,{O_\varphi }(K,{E_0})T{E_0}} \right)\\
&\quad\quad \Longleftrightarrow\quad {{\bar{\tilde V}}_\varphi }\left( {K,{O_\varphi }(K,{TE_0}){TE_0}} \right) \le {{\bar{\tilde V}}_\varphi }\left( {K,{O_\varphi }(K,{E_0})T{E_0}} \right) \\
&\quad\quad \Longleftrightarrow\quad { {\tilde V}_\varphi }\left( {K,{O_\varphi }(K,{TE_0}){TE_0}} \right) \le { {\tilde V}_\varphi }\left( {K,{O_\varphi }(K,{E_0})T{E_0}} \right) \\
&\quad\quad \Longleftrightarrow\quad {O_\varphi }(K,{E_0}) \le {O_\varphi }(K,T{E_0}),
\end{align*}
as desired.
\end{proof}

From Theorem \ref{thm 4.4} and Lemma \ref{lem 4.5}, we can prove the following.
\begin{theorem}\label{thm 4.6}
Suppose $K\in \mathcal{S}^n_o$ and $\varphi\in\Phi$. Then there
exists a unique solution to Problem ${\rm P}_2$.
\end{theorem}

\begin{proof}
First, we prove the existence of solution to problem ${\rm P}_2$.
Observe that  the constraint condition in ${\rm P}_2$ can be turned into  $V(E)=\omega_n$.
Indeed, for any $s\in (0,1)$ and  $E\in \mathcal E^n$ with $V(E)=\omega_n$,
by Lemma \ref{lem 3.4} it follows that
\[{O_\varphi }(K,sE) = {s^{ - 1}}{O_\varphi }(K,E) > {O_\varphi }(K,E),\]
which indicates that $sE$ cannot be a solution to ${\rm P}_2$.

Let ${\lambda _0} = \inf \left\{ {{O_\varphi }(K,TB):T \in {\rm{SL}}(n)} \right\}$.
From Lemma {\ref{lem 3.6}}, we have
\[0 < {\left( {\frac{{V(K)}}{{{\omega _n}}}} \right)^{\frac{1}{n}}} \le {\lambda _0} \le O_\varphi(K,B) < \infty .\]
Similar to the proof of Theorem \ref{thm 4.3}, we can show the set \[\{T\in {\rm SL}(n): O_\varphi(K,TB)\le O_\varphi(K,B)\}\]
is also compact. Combining it with the continuity of $O_\varphi(K,TB)$,   the existence of
solution to ${\rm P}_2$ is demonstrated.

Now, we proceed to prove the uniqueness.

Assume ellipsoid $E_0$ is a solution to ${\rm P}_2$. Then
\[O_\varphi(K,E_0)\le O_\varphi(K,T E_0),\quad {\rm for}\; T\in {\rm SL}(n). \]
By Lemma \ref{lem 4.5}, it follows that
\[{{\tilde V}_\varphi }\left( {K,O_\varphi(K,E_0){E_0}} \right) \le {{\tilde V}_\varphi }\left( {K,O_\varphi(K,E_0)T{E_0}} \right),\quad {\rm for}\; T\in {\rm SL}(n).\]
Thus,  $E_0$ is a solution to Problem ${\rm P}_1$ for star body
$\lambda_0^{-1}K$. Hence, by Theorem \ref{thm 4.4}, the solution to ${\rm P}_2$ is unique.
\end{proof}

\begin{theorem}\label{thm 4.7}
Suppose $K\in \mathcal{S}^n_o$ and $\varphi\in\Phi$. Then

\noindent (1) If  $E_0$  is the unique solution to Problem ${\rm P}_2$, then  $O_\varphi(K,E_0)E_0$ is a solution to Problem ${\rm P}_3$.

\noindent (2) If $E_1$ is a solution to Problem ${\rm P}_3$,  then $\left(\frac{\omega_n}{V(E_1)}\right)^{\frac{1}{n}}E_1$ is a solution to Problem ${\rm P}_2$.

Consequently, there exists a unique solution to Problem ${\rm P}_3$.
\end{theorem}

\begin{proof}
(1) Let $E\in\mathcal{E}^n$ with $O_\varphi(K,E)\le 1$. Since $\left(\frac{\omega_n}{V(E)}\right)^{\frac{1}{n}} E$
satisfies the constraint condition of ${\rm P}_2$, by Lemma \ref{lem 3.4} (2), the fact $V(E_0)=\omega_n$, and the assumption
$O_\varphi(K,E)\le 1$, we have
\begin{align*}
V\left( {{O_\varphi }(K,{E_0}){E_0}} \right) &= {O_\varphi }{(K,{E_0})^n}V({E_0}) \\
&\le {O_\varphi }{\left( {K,{{\left( {\frac{{{\omega _n}}}{{V(E)}}} \right)}^{\frac{1}{n}}}E} \right)^n}V({E_0}) \\
&= \frac{{V(E)}}{{{\omega _n}}}{O_\varphi }{\left( {K,E} \right)^n}V({E_0}) \\
&= V(E){O_\varphi }{\left( {K,E} \right)^n} \\
&\le V(E).
\end{align*}
Thus,
\[V\left( {{O_\varphi }(K,{E_0}){E_0}} \right) \le V(E);\quad {\rm i.e.},\quad \frac{{{\omega _n}}}{V\left( {{O_\varphi
}(K,{E_0}){E_0}} \right) } \ge \frac{{{\omega _n}}}{{V\left( E \right)}},\]
which shows that $O_\varphi(K,E_0)E_0$ solves Problem ${\rm P}_3$.

(2) First, we prove that the constraint condition in ${\rm P}_3$ can
be turned  into $O_\varphi(K,E)=1$; i.e., a solution $E_1$ to ${\rm P}_3$ must satisfies $ O_\varphi(K,E_1)=1$.

Indeed, let  $E\in\mathcal{E}^n$ with $O_\varphi(K,E)<1$. By Lemma \ref{lem 3.4} (2), $O_\varphi(K,O_\varphi(K,E)E)=1$. Since
\[\frac{{{\omega _n}}}{{V\left( {{O_\varphi }(K,E)E} \right)}} = \frac{{{\omega _n}}}{{{O_\varphi }{{(K,E)}^n}V\left( E \right)}} > \frac{{{\omega _n}}}{{V\left( E \right)}},\]
it implies that $E$ cannot be a solution to Problem ${\rm P}_3$.

Now, we can finish the  proof of (2).

Let $E'\in {\mathcal E}^n$ with $V(E')\le\omega_n$. By Lemma \ref{lem 3.4} (2), ${O_\varphi }\left( {K,{O_\varphi }(K,E')E'}\right) = 1$.
Thus  $O_\varphi(K,E')E'$ satisfies the constraint condition of Problem ${\rm P}_3$. Since $E_1$ is a solution to Problem ${\rm P}_3$,
it follows that
\[V\left( O_\varphi(K,E')E' \right) \ge V({E_1}).\]

So, by the assumption $V(E')\le\omega_n$, the fact $O_\varphi(K,E_1)= 1$ and Lemma \ref{lem 3.4} (2), we have
\begin{align*}
{O_\varphi }(K,E') &\ge {\left( {\frac{{V({E_1})}}{{V(E')}}} \right)^{\frac{1}{n}}}
\ge {\left( {\frac{{V({E_1})}}{{{\omega _n}}}} \right)^{\frac{1}{n}}} \\
&= {O_\varphi }(K,{E_1}){\left( {\frac{{V({E_1})}}{{{\omega _n}}}} \right)^{\frac{1}{n}}} \\
&= {O_\varphi }\left( {K,{{\left( {\frac{{{\omega _n}}}{{V({E_1})}}} \right)}^{\frac{1}{n}}}{E_1}} \right).
\end{align*}
Thus,
\[{O_\varphi }(K,E') \ge {O_\varphi }\left( {K,{{\left( {\frac{{{\omega _n}}}{{V({E_1})}}} \right)}^{\frac{1}{n}}}{E_1}} \right),\]
which shows that
$\left(\frac{\omega_n}{V(E_1)}\right)^{\frac{1}{n}}E_1$ solves
Problem ${\rm P}_2$.
\end{proof}

\begin{lemma}\label{lem 4.8}
Suppose $K\in {\mathcal S}^n_o$ and $\varphi\in\Phi$. Then

\noindent(1) $\mathop {\min }\limits_{\left\{ {E \in {{\mathcal E}^n}:  O_\varphi(K,E) \le 1} \right\}} V(E) = \mathop {\min }\limits_{\left\{ {E \in {{\mathcal E}^n}:O_\varphi(K,E) = 1} \right\}} V(E)$.

\noindent (2) $\left\{  E \in {{\mathcal E}^n}: O_\varphi(K,E) = 1 \right\} = \left\{  E \in {{\mathcal E}^n}: {\bar{\tilde V}}_\varphi(K,E) = 1  \right\} $.

\noindent (3) $\mathop {\min }\limits_{\left\{ {E \in {{\mathcal E}^n}:{\bar{\tilde V}}_\varphi(K,E) = 1 } \right\}} V(E) = \mathop
{\min }\limits_{\left\{ {E \in {{\mathcal E}^n}:{\bar{\tilde V}}_\varphi(K,E) \le 1 } \right\}} V(E)$.

Consequently, the solutions to Problems ${\rm P}_3$ and ${\tilde S}_\varphi$ are identical.
\end{lemma}

\begin{proof}
The proof of  (1) can be referred to the proof of  Theorem \ref{thm 4.7} (2).
Assertion (2) follows from Lemma \ref{lem 3.5} directly.

Now, we prove assertion (3).  Let $E\in\mathcal{E}^n$ with ${\bar{\tilde V}}_\varphi(K,E)<1$.
From Definition \ref{def 3.1}  and Lemma \ref{lem 2.1}, we know that the unique positive $\lambda_0$ satisfying the equation
\[\int_{{S^{n - 1}}} {\varphi \left( {\frac{{{\rho _K}}}{{{\lambda _0}{\rho _E}}}} \right)} dV_K^* = \varphi (1)\]
is necessarily  in $(0,1)$, and
\[\int_{{S^{n - 1}}} {\varphi \left( {\frac{{{\rho _K}}}{{\lambda' {\rho _E}}}} \right)} dV_K^* \le \varphi (1),\quad
{\rm i.e.},\quad{{\tilde V}_\varphi }(K,\lambda' E) < 1,\quad {\rm for\; any}\; \lambda'\in (\lambda_0, 1).\]

At the same time, since $V(\lambda' E)<V(E)$,  $\forall \lambda'\in (\lambda_0,1)$, so $E$ cannot possibly solve the minimization problem
\[ \min\left\{V(E): E\in\mathcal{E}^n\;{\rm and}\; {\bar{\tilde  V}}_\varphi(K,E)\le 1\right\}. \]

Hence, assertion (3) is derived.

From the proved (1), (2) and (3), we can conclude that Problem ${\rm P}_3$ and  Problem ${\tilde S}_\varphi$ have  the same solution.
\end{proof}

For different dilations $\lambda_1 K$ and $\lambda_2 K$, $\lambda_1, \lambda_2>0$, 
Problems ${\rm P}_1$  do not generally have the
identical solution.  By contrast, the homogeneity of
$O_\varphi(\lambda K,L)$ in $\lambda\in (0,\infty)$ guarantees that
all Problems ${\rm P}_2$ for $\lambda K$ in $\lambda\in (0,\infty)$
have the identical unique solution. Problems ${\rm P}_3$ and
${\tilde S}_\varphi$ are identical, and Problem ${\rm P}_3$ is the
dual problem of ${\rm P}_2$. Thus, Problem ${\tilde S}_\varphi$ is
not the dual problem of ${\rm P}_1$ in general.

In view of Theorem \ref{thm 4.6}, Theorem \ref{thm 4.7} and Lemma
\ref{lem 4.8}, we are in the position to introduce a family of
ellipsoids in the framework of dual Orlicz Brunn-Minkowski theory,
which are extensions of Legendre ellipsoid.

\begin{definition}\label{def 4.9}
Suppose $K\in\mathcal{S}^n_o$ and $\varphi\in \Phi$. Amongst all
origin-symmetric ellipsoids $E$, the unique ellipsoid that solves
the constrained minimization problem
\[\mathop {\min }\limits_E V(E)\quad {\rm subject\;to}\quad O_\varphi(K,E)\le 1\]
is called the \emph{Orlicz-Legendre ellipsoid} of $K$ with respect
to $\varphi$, and is denoted by ${\rm L}_\varphi K$.

Amongst all origin-symmetric ellipsoids $E$, the unique ellipsoid that solves
the constrained minimization problem
\[\mathop {\min }\limits_E O_\varphi(K,E)\quad{\rm{subject}}\;{\rm{to}}\quad V(E) = {\omega _n} \]
is called the \emph{normalized Orlicz-Legendre ellipsoid} of $K$
with respect to $\varphi$, and is denoted by $\overline{{\rm
L}}_\varphi K$.
\end{definition}

For the polar of ${\rm L}_\varphi K$ or ${\overline {\rm L}}_\varphi K$, 
we write ${\rm L}_\varphi^* K$ or ${\overline {\rm L}}_\varphi^* K$, 
rather than $({\rm L}_\varphi K)^*$ or $({\overline {\rm L}}_\varphi K)^*$.

If $\varphi(t)=t^p$, $1\le p <\infty$, we write ${\rm L}_{\varphi} K$ and
${\overline {\rm L}}_\varphi K$ for ${\rm L}_p K$ and
${\overline {\rm L}}_pK$, respectively. Especially, ${\rm L}_2 K$ is
precisely the Legendre ellipsoid $\Gamma_2 K$.

We observe that for the case $\varphi(t)=t^p$, Problems ${\rm P}_1$
and ${\rm P}_2$ are identical, and were previously solved by Bastero
and Romance  \cite{BR}. Based on their works, Yu \cite{Yu}
introduced the ellipsoids ${\rm L}_p K$ for convex bodies containing
the origin in their interiors.

From Theorem \ref{thm 4.7}, it is obvious that
\begin{equation}\label{representations of O L ellipsoids}
{\rm L}_\varphi K=O_\varphi(K,{{\overline {\rm L}}_\varphi} K)
{{\overline {\rm L}}_\varphi} K\quad {\rm and}\quad {{\overline {\rm
L}}_\varphi} K= \left(\frac{\omega_n}{V({\rm L}_\varphi K)}\right)^{\frac{1}{n}}{\rm L}_\varphi K.
\end{equation}

Definition \ref{def 4.9} combined with inequality (\ref{ineq 3.6.1})
shows that for any $E\in\mathcal{E}^n$,
\[ {\rm L}_\varphi E=E.\]

From Definition \ref{def 4.9} and  Lemma \ref{lem 3.4}, we easily know that the 
operator $L_\varphi$ intertwines with elements of ${\rm GL}(n)$.
\begin{lemma}\label{lem 4.10}
Suppose $K\in\mathcal{S}^n$ and $\varphi\in\Phi$. Then for any $T\in
{\rm GL}(n)$,
\[{\rm L}_\varphi(TK)=T({\rm L}_\varphi K). \]
\end{lemma}

Incidentally, we introduce the following.

\begin{definition}\label{def 4.11}
Suppose $K\in\mathcal{S}^n_o$ and $\varphi\in \Phi$. Amongst all
origin-symmetric ellipsoids $E$, the unique ellipsoid which solves the constrained minimization problem
\[\mathop {\min }\limits_E {{\tilde V}_\varphi }(K,E)\quad{\rm{subject}}\;{\rm{to}}\quad V(E) = {\omega_n} \]
is denoted by ${\rm L}^\diamond_\varphi K$.
\end{definition}

Obviously, if $\varphi(t)=t^p$, $1\le p<\infty$, then 
${\rm L}_\varphi^\diamond K={\overline {\rm L}}_p K$.

\vskip 25pt
\section{\bf The continuity of Orlicz-Legendre ellipsoids}
\vskip 10pt

In this section, we aim to show the continuity of Orlicz-Legendre
ellipsoids ${\rm L}_\varphi K$ with respect to $\varphi$ and $K$.

Throughout this section, we suppose $\varphi\in\Phi$,
$K,K_i\in{\mathcal S}^n_o$, $\varphi,\varphi_j\in\Phi$,
$i,j\in\mathbb N$, and  $K_i\to K$ and $\varphi_j\to\varphi$.
It is easily seen that  there exist positive $r_m$ and $r_M$, such that
\[ r_mB\subseteq K\subseteq r_MB\quad {\rm and}\quad r_mB\subseteq K_i\subseteq r_MB\quad {\rm for \; each}\; i\in\mathbb N. \]

\begin{lemma}\label{lem 5.1}
$\sup \left\{ d\left( {{\rm L}_\varphi ^*K} \right), d\left( {{\rm
L}_\varphi ^*{K_i}} \right),d\left( {{\rm L}_{{\varphi _j}}^*K}
\right),{d\left( {{\rm L}_{{\varphi _j}}^*{K_i}} \right),\;{\rm{with}}\;i,j \in \mathbb{N}} \right\} < \infty$.
\end{lemma}

\begin{proof}
Let $E\in {\mathcal E}^n$. First, we prove the implication

\begin{equation}\label{(5.1)}
O_\varphi(K,E)\le 1 \quad \Longrightarrow\quad d({E^*}) \le \frac{{n{\omega _n}}}{{2{r_m}{\omega _{n - 1}}}}{\varphi ^{ -
1}}\left( {{{\left( {\frac{{{r_M}}}{{{r_m}}}} \right)}^n}\varphi (1)} \right).
\end{equation}

Assume $O_\varphi(K,E)\le 1$. From the definition of
$O_\varphi(K,E)$ together with Lemma \ref{lem 2.1} and Lemma
\ref{lem 3.5},  the definition of ${\tilde V}_\varphi(K,E)$, the
fact $r_mB\subseteq K\subseteq r_MB$ together with the monotonicity
of $\varphi$, the convexity of $\varphi$ together with Jensen's
inequality, (2.3), the fact $h_{E^*}(u)\ge d(E^*)|v_{E^*}\cdot u|$
for $u\in S^{n-1}$, and finally Cauchy's projection formula, it
follows that
\begin{align*}
\varphi (1) &\ge \frac{{{{\tilde V}_\varphi }(K,E)}}{{V(K)}} \\
&= \frac{1}{{nV(K)}}\int_{{S^{n - 1}}} {\varphi \left( {\frac{{{\rho _K}}}{{{\rho _E}}}} \right)\rho _K^ndS}  \\
&\ge {\left( {\frac{{{r_m}}}{{{r_M}}}} \right)^n}\frac{1}{{n{\omega _n}}}\int_{{S^{n - 1}}} {\varphi \left( {\frac{{{r_m}}}{{{\rho _E}}}} \right)dS}  \\
&\ge {\left( {\frac{{{r_m}}}{{{r_M}}}} \right)^n}\varphi \left( {\frac{1}{{n{\omega _n}}}\int_{{S^{n - 1}}} {\frac{{{r_m}}}{{{\rho _E}}}dS} } \right) \\
&= {\left( {\frac{{{r_m}}}{{{r_M}}}} \right)^n}\varphi \left( {\frac{{{r_m}}}{{n{\omega _n}}}\int_{{S^{n - 1}}} {{h_{{E^*}}}dS} } \right) \\
&\ge {\left( {\frac{{{r_m}}}{{{r_M}}}} \right)^n}\varphi \left( {\frac{{{r_m}}}{{n{\omega _n}}}\int_{{S^{n - 1}}} {d({E^*})|{v_{{E^*}}} \cdot u|dS(u)} } \right) \\
&= {\left( {\frac{{{r_m}}}{{{r_M}}}} \right)^n}\varphi \left( {\frac{{2{r_m}{\omega _{n - 1}}}}{{n{\omega _n}}}d({E^*})}\right).
\end{align*}
Thus,
\[\varphi (1) \ge {\left( {\frac{{{r_m}}}{{{r_M}}}} \right)^n}\varphi \left( {\frac{{2{r_m}{\omega _{n - 1}}}}{{n{\omega _n}}}d({E^*})} \right).\]
From the monotonicity of $\varphi$, it yields the inequality in (\ref{(5.1)}).

Since that $\varphi_j\to\varphi$ implies
$\varphi_j(1)\to\varphi(1)$ and $\varphi_j^{-1}\to\varphi^{-1}$, it follows that
\[\varphi _j^{ - 1}\left( {{{\left( {\frac{{{r_M}}}{{{r_m}}}} \right)}^n}{\varphi _j}(1)} \right) \to {\varphi ^{ - 1}}\left( {{{\left( {\frac{{{r_M}}}{{{r_m}}}} \right)}^n}\varphi (1)} \right),\]
and therefore
\[\sup \left\{ {{\varphi ^{ - 1}}\left( {{{\left( {\frac{{{r_M}}}{{{r_m}}}} \right)}^n}\varphi (1)} \right),\varphi _j^{ - 1}\left( {{{\left( {\frac{{{r_M}}}{{{r_m}}}} \right)}^n}{\varphi _j}(1)} \right),{\rm{with}}\;j \in \mathbb{N}} \right\} < \infty .\]
This, as well as (\ref{(5.1)}), proves the desired lemma.
\end{proof}

In light of Lemma \ref{lem 5.1} and Lemma \ref{lem 2.2}, we can show
\begin{lemma}\label{lem 5.2}
$\sup \left\{ {d\left( {{{\overline {\rm L} }_\varphi }K} \right),d\left( {{{\overline {\rm L} }_\varphi }{K_i}}
\right),d\left( {{{\overline {\rm L} }_{{\varphi _j}}}K} \right),d\left({{{\overline {\rm L} }_{{\varphi _j}}}{K_i}} \right),
{\rm{with}}\;i,j \in {\mathbb N}} \right\} < \infty$.
\end{lemma}

\begin{proof}
From (4.3), we have
\begin{align*}
d\left( {\overline {\rm L}}^* _\varphi K  \right) &= d\left( {{{\left( {{{\left( {\frac{{{\omega _n}}}{{V({{\rm L}_\varphi }K)}}} \right)}^{\frac{1}{n}}}{{\rm L}_\varphi }K} \right)}^*}} \right) \\
&= {\left( {\frac{{V({{\rm L}_\varphi }K)}}{{{\omega _n}}}} \right)^{\frac{1}{n}}}d\left( {{\rm L}_\varphi ^*K} \right) \\
&\le {\left( {\frac{{V({{\rm L}_\infty }K)}}{{{\omega _n}}}} \right)^{\frac{1}{n}}}d\left( {{\rm L}_\varphi ^*K} \right) \\
&\le {r_M}d\left( {{\rm L}_\varphi ^*K} \right),
\end{align*}
That is,
\begin{equation}\label{diameter ineq}
d\left( {\overline {\rm L}}^* _\varphi K  \right)\le {r_M}d\left( {{\rm L}_\varphi ^*K} \right).
\end{equation}
Note that Definition \ref{def 6.1}, Theorem \ref{thm 6.2} and the
inequality $V({\rm L}_\varphi K)\le V({\rm L}_\infty K)$ given by
Theorem \ref{thm 8.2},  are previously used here.

Observe that (\ref{diameter ineq}) also holds when
$\varphi$ is replaced by $\varphi_j$ or $K$ is replaced by $K_i$.
Thus, by Lemma \ref{lem 5.1}, it follows
\[\sup \left\{ {d\left( {\overline {\rm L} _\varphi ^*K} \right), d\left( {\overline {\rm L} _\varphi ^*{K_i}} \right), d\left( {\overline {\rm L} _{{\varphi _j}}^*K} \right), d\left( {\overline {\rm L} _{{\varphi _j}}^*{K_i}} \right),\;{\rm{with}}\;i,j \in {\mathbb N}} \right\} < \infty .\]

Take $T_{(0,0)}, T_{(i,j)}, T_{(i,0)}, T_{(0,j)}\in {\rm SL}(n)$,
with $i,j\in\mathbb N$, such that
\[T_{(0,0)}B={\overline {\rm L}}_{\varphi}K,\quad T_{(i,0)}B={\overline {\rm L}}_{\varphi}K_i,\quad T_{(0,j)}B={\overline {\rm L}}_{\varphi_j}K,\quad T_{(i,j)}B={\overline {\rm L}}_{\varphi_j}K_i.\]
Then,
\[\sup \left\{ {\left\| {{T^{-1}_{(0,0)}}} \right\|,\left\| {{T^{-1}_{(i,j)}}} \right\|,\left\| {{T^{-1}_{(i,0)}}} \right\|,\left\| {{T^{-1}_{(0,j)}}} \right\|,\;{\rm{with}}\;i,j \in {\mathbb N}} \right\} < \infty .\]
This, together with Lemma \ref{lem 2.2}, gives
\[\sup \left\{ {\left\| {{T_{(0,0)}}} \right\|,\left\| {{T_{(i,j)}}} \right\|,\left\| {{T_{(i,0)}}} \right\|,\left\| {{T_{(0,j)}}} \right\|,\;{\rm{with}}\;i,j \in {\mathbb N}} \right\} < \infty .\]
Hence, the desired lemma is proved.
\end{proof}

Now, from Lemma \ref{lem 5.2},  there exists a constant $R\in (0,\infty)$, such that all the ellipsoids
${\overline {\rm L}} _\varphi K $, ${\overline {\rm L}} _{\varphi_j} K$, ${\overline {\rm L}}_\varphi K_i$
and ${\overline {\rm L}} _{\varphi_j} K_i$ are in the set
\[\mathcal{E}_R =\left\{E\in\mathcal{E}^n: V(E)=\omega_n\; {\rm and}\; E\subseteq RB  \right\}.\]

From the compactness of the sets $\mathcal{E}_R$ and
$\{K\in\mathcal{S}^n_o: r_m B\subseteq K\subseteq r_MK\}$, together
with Lemma \ref{lem 3.8}, we immediately obtain:
\begin{lemma}\label{lem 5.3}
The limit
$\mathop {\lim }\limits_{i,j \to \infty } {O_{{\varphi _j}}}({K_i},E) = {O_\varphi }(K,E)$
is uniform in $E\in\mathcal{E}_R$.
\end{lemma}

\begin{lemma}\label{lem 5.4}
$\mathop {\lim }\limits_{i,j \to \infty } {O_{{\varphi _j}}}({K_i},{\overline {\rm L} _{{\varphi _j}}}{K_i}) = {O_\varphi }(K,{\overline {\rm L} _\varphi }K)$.
\end{lemma}

\begin{proof}
From Definition \ref{def 4.9} and  Lemma \ref{lem 5.3}, we have
\begin{align*}
\mathop {\lim }\limits_{i,j \to \infty } {O_{{\varphi _j}}}({K_i},{\overline {\rm L} _{{\varphi _j}}}{K_i}) &= \mathop {\lim }\limits_{i,j \to \infty } \mathop {\min }\limits_{E \in {{\mathcal E}_R}} {O_{{\varphi _j}}}({K_i},E) \\
&= \mathop {\min }\limits_{E \in {{\mathcal E}_R}} \mathop {\lim }\limits_{i,j \to \infty } {O_{{\varphi _j}}}({K_i},E) \\
&= \mathop {\min }\limits_{E \in {{\mathcal E}_R}} {O_\varphi }(K,E) \\
&= {O_\varphi }(K,{\overline {\rm L} _\varphi }K),
\end{align*}
as desired.
\end{proof}

\begin{lemma}\label{lem 5.5}
$\mathop {\lim }\limits_{i,j \to \infty } {\overline {\rm L} _{{\varphi _j}}}{K_i} = {\overline {\rm L} _\varphi }K $.
\end{lemma}

\begin{proof}
We argue by contradiction and assume the proposition is false.

Then, from  the compactness of $\mathcal{E}_R$ and  Lemma
\ref{lem 2.3}, there exists a convergent subsequence
$\{{\overline {\rm L}}_{\varphi_{j_q}} K_{i_p} \}_{p,q\in\mathbb N}$, such that
\begin{equation}\label{contradiction}
\mathop {\lim }\limits_{i,j \to \infty } {\overline {\rm L}}
_{{\varphi _j}}{K_i} = {E_0} \in \mathcal{E}_R\quad {\rm{and}}\quad {E_0} \neq {\overline {\rm L}}_\varphi K.
\end{equation}

From Lemma \ref{lem 3.8} and Lemma \ref{lem 5.4}, it follows that
\begin{align*}
{O_\varphi }(K,\mathop {\lim }\limits_{p,q \to \infty } {\overline {\rm L} _{{\varphi _{{j_q}}}}}{K_{{i_p}}}) &= \mathop {\lim }\limits_{p,q \to \infty } {O_\varphi }(K,{\overline {\rm L} _{{\varphi _{{j_q}}}}}{K_{{i_p}}}) \\
&= \mathop {\lim }\limits_{p,q \to \infty } \mathop {\lim }\limits_{k \to \infty } {O_{{\varphi _k}}}(K,{\overline {\rm L} _{{\varphi _{{j_q}}}}}{K_{{i_p}}}) \\
&= \mathop {\lim }\limits_{p,q,k \to \infty } {O_{{\varphi _k}}}(K,{\overline {\rm L} _{{\varphi _{{j_q}}}}}{K_{{i_p}}}) \\
&= \mathop {\lim }\limits_{p,q \to \infty } {O_{{\varphi _{{j_q}}}}}(K,{\overline {\rm L} _{{\varphi _{{j_q}}}}}{K_{{i_p}}}) \\
&= {O_\varphi }(K,{\overline {\rm L} _\varphi }K).
\end{align*}

Since the solution to Problem ${\rm P}_2$ is unique, we have
\[\mathop {\lim }\limits_{p,q \to \infty } {\overline {\rm L} _{{\varphi _{{j_q}}}}}{K_{{i_p}}} = {\overline {\rm L} _\varphi }K,\]
which contradicts (\ref{contradiction}).
\end{proof}

\begin{theorem}\label{thm 5.6}
Suppose $K,K_i\in {\mathcal S}^n_o$ and $\varphi,\varphi_j\in\Phi$, $i,j\in\mathbb N$.
If $K_i\to K$ and $\varphi_j\to\varphi$,
then
\[\mathop {\lim }\limits_{i,j \to \infty } {{\rm L}_{{\varphi _j}}}{K_i} = {{\rm L}_\varphi }K.\]
\end{theorem}

\begin{proof}
From Lemma \ref{lem 5.4}, Lemma \ref{lem 5.5}, together with the identity
\[{{\rm L}_\varphi }K = {O_\varphi }(K,{\overline {\rm L} _\varphi }K){\overline {\rm L} _\varphi }K,\]
the desired limit is immediately derived.
\end{proof}

From Theorem \ref{thm 5.6}, several corollaries are derived directly.
\begin{corollary}
Suppose $K\in {\mathcal S}^n_o$ and $\varphi,\varphi_j\in\Phi$, $j\in\mathbb N$. If  $\varphi_j\to\varphi$, then
\[\mathop {\lim }\limits_{j \to \infty } {{\rm L}_{{\varphi _j}}}{K} = {{\rm L}_\varphi }K.\]
\end{corollary}

\begin{corollary}
Suppose $K,K_i\in {\mathcal S}^n_o$ and $\varphi\in\Phi$, $i\in\mathbb N$. If  $\varphi_i\to\varphi$, then
\[\mathop {\lim }\limits_{i \to \infty } {{\rm L}_\varphi}{K_i} = {{\rm L}_\varphi }K.\]
\end{corollary}

\begin{corollary}
Suppose $K\in\mathcal{S}^n_o$ and $\varphi\in\Phi$. Then ${\rm
L}_{\varphi^p}K$ is continuous in $p\in[1,\infty)$.
\end{corollary}

\begin{corollary}
The $L_p$ Legendre ellipsoid ${\rm L}_pK$ is continuous in
$(K,p)\in{\mathcal S}^n_o\times [1,\infty)$.
\end{corollary}
We observe that although Yu et.al  \cite{Yu} firstly introduced the notion of $L_p$ Legendre ellipsoids,
they did not consider the above continuity at all.

\vskip 25pt
\section{\bf A common limit position}
\vskip 10pt

As Corollary 5.9 claims, for  any $K\in\mathcal{S}^n_o$ and  $\varphi\in\Phi$,
the Orlicz-Legendre ellipsoid ${\rm L}_{\varphi^p}K$ is continuous in $p\in[0,\infty)$.
In this section, we show that  as $p\to \infty$,
${\rm L}_{\varphi^p}K$  approaches to a new ellipsoid ${\rm L}_\infty K$, which is defined by the following.
\begin{definition}\label{def 6.1}
For  $K\in{\mathcal S}^n_o$,  the ellipsoid ${\rm L}_\infty K$  is defined by
\[{\rm L}_\infty K=\left({\rm E}_\infty \left({\rm conv } K\right)^*\right)^* .\]
\end{definition}

Here, ${\rm conv} K$ denotes the convex hull of $K$. Write ${\rm \overline L}_\infty K$ for its normalization, i.e.,
\[ {\rm \overline L}_\infty K =\left(\frac{\omega_n}{V({\rm L}_\infty K)}\right)^{\frac{1}{n}}{\rm L}_\infty K. \]

The following two theorems show a fundamental feature  of $\rm L_\infty K$ and ${\rm \overline L}_\infty K$.

\begin{theorem}\label{thm 6.2}
Suppose $K\in{\mathcal S}^n_o$.  Amongst all origin-symmetric ellipsoids that contain  $K$, the ellipsoid
${\rm L}_\infty K$ is the unique one with minimal volume.
\end{theorem}

For a convex body $K\in\mathcal{K}^n_o$, if the John point of $K^*$ is
at the origin, then $({\rm L}_\infty K)^*$ is precisely the John ellipsoid ${\rm J}(K^*)$
of $K^*$. If $K$ is an origin-symmetric star body in $\mathbb R^n$,
then ${\rm L}_\infty K$ is precisely the L$\rm{\ddot{o}}$wner ellipsoid of $K$.

\begin{proof}
First, observe that for $E\in\mathcal{E}^n$,
\[ K\subseteq E\quad \Longleftrightarrow\quad {\rm conv} K\subseteq E. \]
Indeed, if $K\subseteq E$, then the fact ${\rm conv} E=E$ yields the
inclusion ${\rm conv} K\subseteq E$; conversely, if ${\rm conv} K\subseteq E$, 
then the fact $K\subseteq {\rm conv} K$ yields the inclusion $K\subseteq E$.

Note that ${\rm conv} K\in {\mathcal K}^n_o$. So, for $E\in\mathcal{E}^n$, it holds
\[{\rm conv} K\subseteq E \quad\Longleftrightarrow\quad E^* \subseteq  \left({\rm conv} K \right)^*. \]
From this equivalence and the fact $V(E)V(E^*)=\omega_n^2$, we can reformulate the extremal problem
\[\min \left\{ {V(E):E \in {{\mathcal E}^n}\;{\rm{and}}\;{\rm{conv}}K \subseteq E} \right\}\]
equivalently as
\[\max \left\{ {V({E^*}):E \in {{\mathcal E}^n}\;{\rm{and}}\;{E^*} \subseteq {{({\rm{conv}}K)}^*}} \right\}.\]

Recall that the John ellipsoid ${\rm E}_\infty \left({\rm conv } K\right)^*$ \cite{LYZ6, ZouXiong} is
the unique solution to the above maximization problem. Since ${\rm E}_\infty \left({\rm conv } K\right)^*$
is the unique origin-symmetric ellipsoid of maximal volume contained in the convex body $\left({\rm conv } K\right)^*$,
we know that $\left({\rm E}_\infty \left({\rm conv } K\right)^*\right)^*$ is the unique ellipsoid of minimal volume
containing ${\rm conv} K$.
\end{proof}

\begin{theorem}\label{thm 6.3}
Suppose $K\in\mathcal {S}^n_o$. Amongst all origin-symmetric ellipsoids $E$, the ellipsoid ${\rm \overline L}_\infty K$ uniquely solves the constrained minimization problem
\[\mathop {\min }\limits_E {\left\| {\frac{{{\rho _K}}}{{{\rho _E}}}}
\right\|_\infty }\quad{\rm{subject}}\;{\rm{to}}\quad V(E)\le\omega_n. \]
\end{theorem}

\begin{proof}
The proof will be complete after two steps.

First, we show that the ellipsoid ${\rm \overline L}_\infty K$ solves the desired extremal problem.

Let $E\in {\mathcal E}^n$ with $V(E)\le\omega_n$. From the identity
$\left\| \rho_K/\rho_{\|\rho_K/\rho_E\|_\infty E} \right\|_\infty=1$
and the implication
\[{\left\| \frac{\rho _K}{\rho _L} \right\|_\infty } = 1\quad \Longrightarrow\quad L \supseteq K,\quad {\rm for}\;L \in S_o^n,\]
it follows that $\left\|\rho_K/\rho_E \right\|_\infty E\supseteq K$.
Thus, by Theorem \ref{thm 6.2},
\[V\left( \left\|\rho_K/\rho_E \right\|_\infty E  \right) \ge V(\rm L _\infty K).\]
From this inequality, the assumption that $V(E)\le \omega_n$, the fact that $\|\rho_K/\rho_{\rm L_\infty K}\|_\infty = 1$,  and finally the
definition of ${\rm \overline L}_\infty K$, it follows that
\begin{align*}
{\left\| {\frac{{{\rho _K}}}{{{\rho _E}}}} \right\|_\infty } &\ge {\left( {\frac{{V({\rm L_\infty }K)}}{{V(E)}}} \right)^{\frac{1}{n}}} \\
&\ge {\left( {\frac{{V({\rm L_\infty }K)}}{{{\omega _n}}}} \right)^{\frac{1}{n}}} \\
&= {\left\| {\frac{{{\rho _K}}}{{{\rho _{{\rm L_\infty }K}}}}} \right\|_\infty }{\left( {\frac{{V({\rm L_\infty }K)}}{{{\omega _n}}}} \right)^{\frac{1}{n}}} \\
&= {\left\| {\frac{{{\rho _K}}}{{{\rho _{{{\left( {\frac{{{\omega _n}}}{{V({\rm L_\infty }K)}}} \right)}^{\frac{1}{n}}}{\rm L_\infty }K}}}}} \right\|_\infty } \\
&= {\left\| {\frac{{{\rho _K}}}{{{\rho _{{{\rm \overline L }_\infty }K}}}}} \right\|_\infty}.
\end{align*}
That is,
\[{\left\| {\frac{{{\rho _K}}}{{{\rho _E}}}} \right\|_\infty } \ge {\left\| {\frac{{{\rho _K}}}{{{\rho _{{{\rm \overline L }_\infty }K}}}}} \right\|_\infty },\]
which implies that ${\rm \overline L}_\infty K$ is a solution to the desired extremal problem.

Assume that $E_0$ is a solution to the considered extremal problem. Now,  we aim to show that
${\left\| {\frac{{{\rho _K}}}{{{\rho _{{E_0}}}}}} \right\|_\infty }{E_0}$ is an
origin-symmetric ellipsoid of minimal volume containing $K$.
If so, according to the uniqueness of $\rm L_\infty K$, we  obtain that ${\rm \overline L}_\infty K$ is the unique solution
to the considered problem.

Let $E'\in\mathcal{E}^n$ with $K\subseteq E'$. From the facts that
$1 \ge {\left\|\rho_K/\rho_{E'} \right\|_\infty }$ and $V(E_0)=\omega_n$, it follows that
\[V(E') \ge \left\| {\frac{{{\rho _K}}}{{{\rho _E}'}}} \right\|_\infty ^nV(E') = \frac{{V({E_0})}}{{{\omega _n}}}\left\| {\frac{{{\rho _K}}}{{{\rho _{E'}}}}} \right\|_\infty ^nV(E') = \left\| {\frac{{{\rho _K}}}{{{\rho _{{{\left( {\frac{{{\omega _n}}}{{V(E')}}} \right)}^{\frac{1}{n}}}E'}}}}} \right\|_\infty ^nV({E_0}).\]
Especially, we have
\[V(E') \ge \left\| {\frac{{{\rho _K}}}{{{\rho _{{E_0}}}}}} \right\|_\infty ^nV({E_0}) = V\left( {{{\left\| {\frac{{{\rho _K}}}{{{\rho _{{E_0}}}}}} \right\|}_\infty }{E_0}} \right).\]
Note that $K \subseteq {\left\| {\frac{{{\rho _K}}}{{{\rho _{{E_0}}}}}} \right\|_\infty }{E_0}$. 
Thus, the ellipsoid ${\left\| {\frac{{{\rho _K}}}{{{\rho _{{E_0}}}}}} \right\|_\infty }{E_0}$ is
an origin-symmetric ellipsoid of minimal volume containing $K$.
\end{proof}

Now, we turn to the main result in this section.
\begin{theorem}\label{thm 6.4}
Suppose $K\in\mathcal{S}^n_o$ and $\varphi\in\Phi$. Then 
$\mathop {\lim }\limits_{p \to \infty } {\rm L}_{{\varphi ^p}}K = {\rm L_\infty }K$.
\end{theorem}

From the arguments in Section 5, we know that the set $\{ {\rm \overline L}_{\varphi^p}K: 1\le p <\infty\}$ is bounded from above. Hence,
there exists a  constant $C\in (0,\infty)$ such that
\begin{align*}
&\{{\rm \overline L}_{\varphi^p}K: 1\le p <\infty\}\cup\{{\rm \overline L}_\infty K\}\\ &\quad \subseteq \mathcal{F}=\{E\in\mathcal{E}^n: V(E)=\omega_n\;
{\rm and}\; E\subseteq CB\}.
\end{align*}

For $p\in [1,\infty)$, define the functional $f_p: \mathcal{F}\rightarrow (0,\infty)$ by
\[ {f_p}(E) =O_{{\varphi ^p}}(K,E), \quad {\rm for}\;E \in \mathcal{F}, \]
and the functional $f_\infty: \mathcal{F}\rightarrow  (0,\infty)$ by
\[{f_\infty }(E) = {\left\| {\frac{{{\rho _K}}}{{{\rho _E}}}} \right\|_\infty }, \quad {\rm for}\;E \in \mathcal{F}.\]

To prove Theorem \ref{thm 6.4}, several lemmas are in order.

First, applying Lemma \ref{lem 3.9} (4) to the functionals $f_j$ and $f_\infty$ on $\mathcal{F}$, we have
\begin{lemma}\label{lem 6.5}
$\mathop {\lim }\limits_{j \to \infty } {f_j}(E) = {f_\infty }(E), $ for $E\in\mathcal{F}$.
\end{lemma}

\begin{lemma}\label{lem 6.6}
The limit $\mathop {\lim }\limits_{j \to \infty } {f_j}(E) = {f_\infty }(E)$ is uniform in $E\in\mathcal{F}$.
\end{lemma}

\begin{proof}
We argue by contradiction and assume the conclusion to be false.
By our assumption, the definitions of $f_{j_k}$ and $f_\infty$
together with Lemma \ref{lem 3.9} (3), there exist an
$\varepsilon_0>0$, a  sequence $\{j_k\}_k$ strictly increasing to
$\infty$, and a sequence $E_k\subset \mathcal{F}$, such that
\[|f_\infty(E_k)-f_{j_k}(E_k)|>\varepsilon_0;\quad {\rm i.e.},\quad {f_{{j_k}}}({E_k}) < {f_\infty }({E_k}) - {\varepsilon _0},\quad {\rm for}\; k\in\mathbb N.\]
Thus, these inequalities together with Lemma \ref{lem 3.9} (3)  yield that
\[{f_i}({E_k}) < {f_\infty }({E_k}) - {\varepsilon _0},\quad {\rm{for}}\;i \le j_k\; {\rm and}\; k\in\mathbb N.\]
Meanwhile, from the compactness of $(\mathcal{F},\delta_H)$ together with Lemma \ref{lem 2.3}, there
exists a convergent subsequence $\{E_{k_l}\}_l$ of $\{E_k\}_k$, which converges to certain $E_0\in\mathcal{F}$.

Consequently, letting $l\to\infty$ in the inequality
\[{f_i}({E_{{k_l}}}) < {f_\infty }({E_{{k_l}}}) - {\varepsilon _0},\quad {\rm{for}}\;i \le k_l\; {\rm and}\; l\in\mathbb N,\]
and using the continuity of $f_i$ and $f_\infty$, we have
\[{f_i}({E_0}) \le {f_\infty }({E_0}) - {\varepsilon _0},\quad {\rm{for}}\;i \in \mathbb N,\]
which contradicts Lemma \ref{lem 6.5}.
\end{proof}

Using Lemma \ref{lem 6.6}, we can prove the following.
\begin{lemma}\label{lem 6.7} $\mathop {\lim }\limits_{p \to \infty } {\rm \overline L} _{{\varphi ^p}}K = {\rm \overline L _\infty }K$.
\end{lemma}

\begin{proof}
By the boundedness of $\{{\rm \overline L}_{\varphi^p}K: 1\le p<\infty\}$ in $(\mathcal{F},\delta_H)$,
it suffices to prove that
\[\mathop {\lim }\limits_{j \to \infty } {\rm \overline L} _{{\varphi ^{{p_j}}}}K = {\rm \overline L _\infty }K,\]
for any convergent subsequence $\{{\rm \overline L}_{\varphi^{p_j}}K\}_j$
with $p_j$ strictly increasing to $\infty$.

Assume that  $\mathop {\lim }\limits_{j \to \infty } {\rm \overline L _{{\varphi ^{{p_j}}}}}K = {E_0}$.
From the definition of $f_\infty$,  the continuity of $f_\infty$, Lemma \ref{lem 6.5}, and Lemma \ref{lem 6.6}, we have
\begin{align*}
{\left\| {\frac{{{\rho _K}}}{{{\rho _{{E_0}}}}}} \right\|_\infty } &= {f_\infty }\left( {{E_0}} \right) \\
&= {f_\infty }\left( {\mathop {\lim }\limits_{j \to \infty } {{\rm \overline L }_{{\varphi ^{{p_j}}}}}K} \right) \\
&= \mathop {\lim }\limits_{j \to \infty } {f_\infty }\left( {{{\rm \overline L }_{{\varphi ^{{p_j}}}}}K} \right) \\
&= \mathop {\lim }\limits_{j \to \infty } \mathop {\lim }\limits_{i \to \infty } {f_{{p_i}}}\left( {{{\rm \overline L }_{{\varphi ^{{p_j}}}}}K} \right) \\
&= \mathop {\lim }\limits_{i,j \to \infty } {f_{{p_i}}}\left( {{{\rm \overline L }_{{\varphi ^{{p_j}}}}}K} \right) \\
&= \mathop {\lim }\limits_{j \to \infty } {f_{{p_j}}}\left( {{{\rm \overline L }_{{\varphi ^{{p_j}}}}}K} \right).
\end{align*}

Moreover, from the definition of ${\rm \overline L}_{\varphi^{p_j}}K$
together with the fact that ${\rm \overline L}_{\varphi^{p_j}}K\in {\mathcal F}$, Lemma \ref{lem 6.6} together with the compactness of
$\mathcal{F}$, Lemma \ref{lem 6.5}, and  the definition of $f_\infty$, it follows 
\begin{align*}
{\left\| {\frac{{{\rho _K}}}{{{\rho _{{E_0}}}}}} \right\|_\infty } &= \mathop {\lim }\limits_{j \to \infty } \mathop {\min }\limits_{E \in \mathcal{F}} {f_{{p_j}}}\left( E \right) \\
&= \mathop {\min }\limits_{E \in \mathcal{F}} \mathop {\lim }\limits_{j \to \infty } {f_{{p_j}}}\left( E \right) \\
&= \mathop {\min }\limits_{E \in \mathcal{F}} {f_\infty }\left( E \right) \\
&= \mathop {\min }\limits_{E \in \mathcal{F}} {\left\| {\frac{{{\rho _K}}}{{{\rho _E}}}} \right\|_\infty }.
 \end{align*}
Thus,
\[{\left\| {\frac{{{\rho _K}}}{{{\rho _{{E_0}}}}}} \right\|_\infty } = \mathop {\min }\limits_{E \in \mathcal{F}} {\left\| {\frac{{{\rho _K}}}{{{\rho _E}}}} \right\|_\infty },\]
From the fact that ${\rm \overline L}_\infty K\in\mathcal{F}$ and the uniqueness of ${\rm \overline L}_\infty K$,
it yields that ${{E_0}}={\rm \overline L}_\infty K$. 
\end{proof}

\begin{lemma}\label{lem 6.8}
$\mathop {\lim }\limits_{p \to \infty } O_{\varphi^p}(K,{\rm \overline L}_{\varphi^p}K)= {\left\| {\frac{{{\rho _K}}}{{{\rho _{{{\rm \overline L }_\infty }K}}}}} \right\|_\infty }$.
\end{lemma}
\begin{proof}
From the definition of $f_p$,  Lemma \ref{lem 6.6}, Lemma \ref{lem 6.7}, and the definition of $f_\infty$, it follows that
\begin{align*}
\mathop {\lim }\limits_{p \to \infty } O_{\varphi^p}(K,{\rm \overline L}_{\varphi^p}K)
&= \mathop {\lim }\limits_{p \to \infty } {f_p}\left( {{{\rm \overline L }_{{\varphi ^p}}}K} \right) \\
&= \left( {\mathop {\lim }\limits_{p \to \infty } {f_p}} \right)\left( {\mathop {\lim }\limits_{p \to \infty } {{\rm \overline L }_{{\varphi ^p}}}K} \right) \\
&= {f_\infty }\left( {{{\rm \overline L }_\infty }K} \right) \\
&= {\left\| {\frac{{{\rho _K}}}{{{\rho _{{{\rm \overline L }_\infty }K}}}}} \right\|_\infty},
\end{align*}
as desired.
\end{proof}

Now, we are in the position to finish the proof of  Theorem \ref{thm 6.4}.
\begin{proof}[Proof of Theorem 6.4]
From the identities
\[O_{{\varphi ^p}}(K,{\rm \overline L}_{\varphi^p}K){\rm \overline L} _{{\varphi ^p}}K = {\rm L}_{{\varphi ^p}}K \quad {\rm and}\quad
{\left\| {\frac{{{\rho _K}}}{{{\rho _{{{\rm \overline L }_\infty }K}}}}} \right\|_\infty }{\rm \overline L _\infty }K = {\rm L_\infty }K,\]
together with Lemmas \ref{lem 6.7} and  \ref{lem 6.8}, Theorem \ref{thm 6.4} is derived immediately.
\end{proof}

Note that if $K$ is an origin-symmetric star body in $\mathbb R^n$,
then the Orlicz-Legendre ellipsoid ${\rm L}_{\varphi^p} K$ converges to
the L${\rm \ddot{o}}$wner ellipsoid ${\rm L}K$ as $p\to\infty$.

\vskip 25pt
\section{\bf A Characterization of Orlicz-Legendre ellipsoid}
\vskip 10pt

In this section, we  establish  a  connection linking the characterization  of  
Orlicz-Legendre ellipsoids and the isotropy of measures.

\begin{definition}\label{def 7.1}
Suppose $K\in\mathcal{S}^n_o$ and $\varphi\in\Phi\cap C^1[0,\infty)$,
the Borel measure $\mu_\varphi(K,\cdot)$ on $S^{n-1}$ is defined by
\[ d\mu_\varphi(K,\cdot)=\varphi'\left(\rho_K\right)\rho_K^{n+1}dS. \]
\end{definition}

The next theorem  not only characterizes the ellipsoid ${\rm L}_\varphi^\diamond K$, but also plays a crucial role to
establish Theorem \ref{thm 7.4}.

\begin{theorem}\label{thm 7.2}
Suppose $K\in\mathcal{S}^n_o$ and $\varphi\in\Phi\cap C^1[0,\infty)$. 
Then, ${\rm L}_\varphi^\diamond K=B$, if and only if the
measure $\mu_\varphi(K,\cdot)$ is isotropic on $S^{n-1}$, i.e.,
\[\frac{n}{{|{\mu _\varphi }(K, \cdot )|}}\int_{{S^{n - 1}}} {u \otimes ud{\mu _\varphi }(K,u)}  = {I_n}.\]
\end{theorem}

\begin{proof}
First, We show the necessity by variational method.

Let $L: \mathbb R^n\to\mathbb R^n$ be a linear transformation.
Choose $\varepsilon_0>0$ sufficiently small so that for all
$\varepsilon\in (-\varepsilon_0,\varepsilon_0)$ the matrix
$I_n+\varepsilon L$ is invertible. For $\varepsilon\in(-\varepsilon_0,\varepsilon_0)$, define
\[ L_\varepsilon=\frac{I_n+\varepsilon L}{|I_n+\varepsilon L|^{\frac{1}{n}}}. \]
Then $L_\varepsilon\in {\rm SL}(n)$. The assumption that
$L_\varphi^\diamond K=B$ implies that for all $\varepsilon$,
\[ \tilde{V}_\varphi(K,L_\varepsilon^{-1}B)\ge {\tilde V}_\varphi(K,B). \]

The fact $\frac{1}{{{\rho _{L_\varepsilon ^{ - 1}B}}(u)}} = {h_{L_\varepsilon ^tB}}(u)$ for $u\in S^{n-1}$, together with the
definition of $\tilde{V}_\varphi(K,L_\varepsilon^{-1}B)$, gives
\[{{\tilde V}_\varphi }(K,L_\varepsilon ^{ - 1}B) = \int\limits_{{S^{n - 1}}} {\varphi \left( {{\rho _K}(u)\frac{{{{(1 + 2\varepsilon u \cdot Lu + {\varepsilon ^2}Lu \cdot Lu)}^{\frac{1}{2}}}}}{{|{I_n} + \varepsilon L{|^{\frac{1}{n}}}}}} \right)d{{\tilde V}_K(u)}} .\]
From the smoothness of $\varphi$ and
$|L_\varepsilon u|$ in $\varepsilon$, the integrand depends smoothly
on $\varepsilon$. Thus,
\[{\left. {\frac{d}{{d\varepsilon }}} \right|_{\varepsilon  = 0}}{{\tilde V}_\varphi }(K,L_\varepsilon ^{ - 1}B) = 0.\]

Calculating it directly, we have
\begin{align*}
0 &= \int_{{S^{n - 1}}} {{{\left. {\frac{\partial }{{\partial \varepsilon }}} \right|}_{\varepsilon  = 0}}\varphi \left( {{\rho _K}(u)\frac{{{{(1 + 2\varepsilon u \cdot Lu + {\varepsilon ^2}Lu \cdot Lu)}^{\frac{1}{2}}}}}{{|{I_n} + \varepsilon L{|^{\frac{1}{n}}}}}} \right)d{{\tilde V}_K}(u)}  \\
&= \int_{{S^{n - 1}}} {\varphi '\left( {{\rho _K}(u)} \right)\left( { - \frac{{{\mathop{\rm tr}\nolimits} \;L}}{n} + u \cdot Lu} \right){\rho _K}(u)d{{\tilde V}_K}(u)}  \\&= \frac{1}{n}\int_{{S^{n - 1}}} {\left( { - \frac{{{\mathop{\rm tr}\nolimits} \;L}}{n} + u \cdot Lu} \right)d{\mu _\varphi }(K,u)} .
\end{align*}

Let $v\in S^{n-1}$ and $L=v\otimes v$. Using the facts ${\rm tr} (v\otimes v) =1$ and $u\cdot (v\otimes v) u= (u\cdot v)^2$, it gives
\[\int_{{S^{n - 1}}} {{{(u \cdot v)}^2}d{\mu _\varphi }(K,u)} =  \frac{{|{\mu _\varphi }(K, \cdot )|}}{n}.\]
Thus,  $\mu_\varphi(K,\cdot)$ is isotropic on $S^{n-1}$.

Next, we prove the sufficiency. Suppose that $\mu_\varphi(K,\cdot)$ is isotropic on $S^{n-1}$.
It  suffices to  prove that if $E\in\mathcal{E}^n$ and $V(E)=\omega_n$, then
\[ {\tilde V}_\varphi(K,E)\ge {\tilde V}_\varphi(K,B), \]
If so, it  will imply that ${\rm L}_\varphi^\diamond K= B$. The proof will be completed after three steps.

First, for $a=(a_1,\cdots,a_n)\in [0,\infty)^n$, define
\[F(a) = \int_{{S^{n - 1}}} {\varphi \left( {{\rho _K}(u)} \right)|{\mathop{\rm diag}\nolimits} ({a_1}, \cdots ,{a_n})u|d{{\tilde V}_K}(u)} ,\]
where ${\rm diag}(a_1,\cdots,a_n)$ denotes the $n\times n$ diagonal
matrix with diagonal elements $a_1,\cdots,a_n$.

We aim to show  that
\begin{equation}\label{step1 for thm 8.2}
F(a)\ge F(e),\quad {\rm whenever}\; \prod_{j=1}^na_j=1.
\end{equation}
Here, $e$ denotes the point $(1,\cdots,1)$.

From the smoothness of $\varphi$ and  $|{\rm diag}(a_1,\cdots,a_n) u|$ in $(a_1,\cdots,a_n)$, we have
\begin{align*}
{\left.{\frac{\partial }{{\partial {a_j}}}} \right|_{a = e}}F(a) &= \int_{{S^{n - 1}}} {{{\left.{\frac{\partial }{{\partial {a_j}}}} \right|}_{a = e}}\varphi \left( {{\rho _K}(u)} \right)|{\mathop{\rm diag}\nolimits} ({a_1}, \cdots ,{a_n})u|d{{\tilde V}_K}(u)}  \\
&= \int_{{S^{n - 1}}} {\varphi '\left( {{\rho _K}(u)} \right){\rho _K}(u){{\left. {\frac{\partial }{{\partial {a_j}}}} \right|}_{a = e}}|{\mathop{\rm diag}\nolimits} ({a_1}, \cdots ,{a_n})u|d{{\tilde V}_K}(u)}  \\
&= \int_{{S^{n - 1}}} {u_j^2\varphi '\left( {{\rho _K}(u)} \right){\rho _K}(u)d{{\tilde V}_K}(u)},
\end{align*}
where $(u_1,\cdots,u_n)$ denotes the coordinates of $u\in S^{n-1}$.
From the isotropy  of $\mu_\varphi(K,\cdot)$, it follows that
\[{\left. {\frac{\partial }{{\partial {a_j}}}} \right|_{a = e}}F(a) = \frac{{|{\mu _\varphi }(K, \cdot )|}}{n}.\]
Thus,
\begin{equation}\label{outer normal}
\nabla F(e) = \frac{{|{\mu _\varphi }(K, \cdot )|}}{n}e.
\end{equation}

It can be checked  that the function $F:  [0, \infty)^n\to [0,\infty)$ is
continuous and convex, and $F(\lambda a)$ is strictly increasing in
$\lambda\in [0,\infty)$,  for  $a\in {(0,\infty )^n}$. Thus,
$F^{-1}([0, F(e)])$ is compact, convex and of non-empty interior.
Precisely, it is a convex body. Its boundary is given by the
equation $ F(a)=F(e)$ with $a\in [0,\infty)^n$, so (\ref{outer
normal}) implies the vector $e$ is an outer normal of the convex
body $F^{-1}([0, F(e)])$ at the boundary point $e$.

Consequently, ${F^{ - 1}}([0,F(e)]) \subset \left\{ {a \in {\mathbb
R^n}:a \cdot e \le n} \right\}$. That is to say, for all
$a\in\mathbb [0,\infty)^n$, if $F(a)\le F(e)$, then $a\cdot e\le n$.
In contrast, for all $b=(b_1,\cdots,b_n)\in (0,\infty)^n$ with
$b_1\cdots b_n =1$, the AM-GM inequality yields that $b\cdot e\ge
n$, with equality if and only if $b=e$. Hence,  (\ref{step1 for thm 8.2}) is derived.

Secondly, with (\ref{step1 for thm 8.2}) in hand, we aim to show that for
$T\in {\rm SL}(n)$,
\begin{equation}\label{step2 for thm 8.2}
{\tilde{V}_\varphi }(TK,B) \ge {{\tilde V}_\varphi }(K,B),
\end{equation}
with equality if and only if $T$ is orthogonal.

Indeed, it is known that each $T\in{\rm SL}(n)$ can be represented as
$T^{-1}=O_1^{-1}AO_2$, where $O_1$, $O_2$ are $n\times n$
orthogonal matrices, and  $A={\rm diag}(a_1,\cdots, a_n)$ is  diagonal and positive definite with $a_1a_2\cdots a_n=1$. Note that
${{\tilde V}_\varphi }(TK,B) = {{\tilde V}_\varphi }(O_1K,AB)$. So,
applying (\ref{step1 for thm 8.2}) to the body $O_1K$, it gives
(\ref{step2 for thm 8.2}).

Finally, we rewrite inequality (\ref{step2 for thm 8.2}) equivalently as
\[{\tilde{V}_\varphi }(K,E) \ge {\tilde{V}_\varphi }(K,B),\]
for all $E\in\mathcal{E}^n$ with $V(E)=\omega_n$, with equality if
and only if $E=B$. So, ${\rm L}_\varphi^\diamond K=B$.

The proof is complete.
\end{proof}

\begin{corollary}\label{cor 7.3}
Suppose $K\in\mathcal{S}^n_o$ and $\varphi\in\Phi\cap C^1[0,\infty)$.
Then, modulo orthogonal transformations, there exists an ${\rm SL}(n)$ transformation $T$
such that the measure $\mu_\varphi(TK,\cdot)$ is isotropic on $S^{n-1}$.
\end{corollary}

\begin{theorem}\label{thm 7.4}
Suppose $K\in\mathcal{S}^n_o$, $\varphi\in\Phi\cap C^1[0,\infty)$
and $T\in {\rm SL}(n)$. Then the following assertions are
equivalent:

\noindent (1) ${\rm L}_\varphi\left(O_\varphi(K,TB)^{-1}K\right)=TB$,

\noindent (2) ${\rm \overline L}_\varphi K= TB$,

\noindent (3) ${\rm L}_\varphi^\diamond\left(O_\varphi(T^{-1}K,B)^{-1}T^{-1}K\right)=B$,

\noindent (4) $\mu_\varphi\left(O_\varphi(T^{-1}K,B)^{-1}T^{-1}K,\cdot\right)$ is isotropic on $S^{n-1}$.
\end{theorem}

\begin{proof}
Equations (\ref{representations of O L ellipsoids}) yields the
equivalence ``(1) $\Leftrightarrow$ (2)".  Combining Lemma \ref{lem 4.5} with Lemma \ref{lem 3.4} (1),
it gives the equivalence ``(2) $\Leftrightarrow$ (3)".
Finally, Theorem \ref{thm 7.2} implies the equivalence ``(3) $\Leftrightarrow$ (4)".
\end{proof}

\vskip 25pt
\section{Volume ratio inequalities}
\vskip 10pt

In general, the Orlicz-Legendre ellipsoid ${\rm L}_\varphi K$ does not
contain $K$. However, we show that the volume functional over the class of Orlicz-Legendre ellipsoids of $K$
is bounded by $V({\rm L}_1K)$ from below and by $V({\rm L}_\infty K)$ from above.
\begin{theorem}\label{thm 8.1}
Suppose $K\in\mathcal{S}^n_o$, $\varphi\in\Phi$ and $1\le p< q<\infty$. Then
\[ V({\rm L}_1K)\le V({\rm L}_\varphi K)\le V({\rm L}_{\varphi^p} K)\le V({\rm L}_{\varphi^q} K)\le V({\rm L}_\infty K). \]
\end{theorem}

\begin{proof}
From Lemma \ref{lem 3.9}, it follows that
\begin{align*}
\left\{ {E \in {\mathcal{E}^n}:{{\left\| {\frac{{{\rho _K}}}{{{\rho _E}}}:{V^*_K}} \right\|}_1} \le 1} \right\} &\supseteq \left\{ {E \in {\mathcal{E}^n}:{{\left\| {\frac{{{\rho _K}}}{{{\rho _E}}}:{V^*_K}} \right\|}_\varphi } \le 1} \right\} \\
&\supseteq \left\{ {E \in {\mathcal{E}^n}:{{\left\| {\frac{{{\rho _K}}}{{{\rho _E}}}:{V^*_K}} \right\|}_{{\varphi ^p}}} \le 1} \right\} \\
&\supseteq \left\{ {E \in {\mathcal{E}^n}:{{\left\| {\frac{{{\rho _K}}}{{{\rho _E}}}:{V^*_K}} \right\|}_{{\varphi ^q}}} \le 1} \right\} \\
&\supseteq \left\{ {E \in {\mathcal{E}^n}:{{\left\| {\frac{{{\rho
_K}}}{{{\rho _E}}}} \right\|}_\infty } \le 1} \right\}.
 \end{align*}
From the above inclusions and the definition of Orlicz-Legendre ellisoids,
the desired inequalities are obtained.
\end{proof}

\begin{theorem}\label{thm 8.2}
Suppose $K\in\mathcal{S}^n_o$ and $\varphi\in\Phi$. Then
\[V({\rm L}_\varphi K)\ge V(K), \]
with equality if and only if $K\in\mathcal{E}^n$.
\end{theorem}

\begin{proof}
From Lemma \ref{lem 3.6},  it follows that
\[{O_\varphi }\left( {K,{{\rm L}_\varphi }K} \right) \ge {\left( {\frac{{V(K)}}{{V({{\rm L}_\varphi }K)}}} \right)^{\frac{1}{n}}},\]
with equality if and only if $K\in\mathcal{E}^n$. Combing this with the fact
\[1 = {O_\varphi }\left( {K,{{\rm L}_\varphi }K} \right),\]
the desired inequality is followed.
\end{proof}

If $\varphi(t)=t^p$, $1\le p<\infty$, then
Theorem \ref{thm 8.2} implies that $V({\rm L}_pK)\ge V(K)$, and in particular that $V(\Gamma_2K)\ge V(K)$.

A classical result on John's ellipsoid
is Ball's volume ratio inequality \cite{Ball1, Ball2}, which states: if $K$ is an
origin-symmetric convex body in $\mathbb R^n$, then
\[\frac{{V(K)}}{{V({\rm J}K)}} \le \frac{{{2^n}}}{{{\omega _n}}},\]
with equality if and only if $K$ is a parallelotope. The fact that
equality holds in Ball's inequality only for parallelotope was
established by Barthe \cite{Barthe}. He also established the outer
volume-ratio inequality: if $K$ is an origin-symmetric convex body
in $\mathbb R^n$, then
\[\frac{{V(K)}}{{V({\rm L}K)}} \ge \frac{{{2^n}}}{{n!{\omega _n}}},\]
with equality if and only if $K$ is a cross-polytope.

Recall that when $K$ is an origin-symmetric convex body,
${\rm L}_\infty K$ is just the L${\rm \ddot{o}}$wner ellipsoid ${\rm L}K$.
Thus,  Combining Theorem \ref{thm 8.1} with Barthe's outer volume ratio
inequality, we immediately obtain

\begin{theorem}\label{thm 8.3}
Suppose $K\in\mathcal{K}^n_o$ is origin-symmetric and $\varphi\in\Phi$. Then
\[\frac{{V(K)}}{{V({{\rm L}_\varphi }K)}} \ge \frac{{{2^n}}}{{n!{\omega _n}}}.\]
\end{theorem}

It is easily seen that the volume ratio $\frac{V({\rm L}_\varphi K)}{V(K)}$ is
${\rm GL}(n)$-invariant and  minimized by origin-symmetric
ellipsoids.  Theorem \ref{thm 8.3}  shows that
$\frac{V({\rm L}_\varphi K)}{V(K)}$ is bounded from above.  However,
the exact equality condition is not yet known.

\noindent \textbf{Problem.} \emph{Suppose $\varphi\in\Phi$. Amongst
all origin-symmetric convex bodies $K$ in $\mathbb R^n$,   which
ones maximize the volume ratio $\frac{V({\rm L}_\varphi K)}{V(K)}$? }

A particular case concerns with the volume ratio $\frac{V(\Gamma_2 K)}{V(K)}$.
As pointed out by Schneider \cite{Schneider} and LYZ \cite{LYZ1},
that to find the maximizers for $\frac{V(\Gamma_2 K)}{V(K)}$ over the class of  origin-symmetric convex
bodies is still a major open problem in convex geometry.
It is even difficult to show that there exists a constant $c$ which is independent of the
dimension $n$ and bounds the volume ratio $\frac{V(\Gamma_2 K)}{V(K)}$ from above. This problem was
firstly posed by Bourgain \cite{Bourgain1}. For more  information, we refer to
Bourgain \cite{Bourgain2}, Dar \cite{Dar}, Junge \cite{Junge}, Lindenstrauss and Milman \cite{LindenstraussMilman},
LYZ \cite{LYZ1}, and Milman and Pajor \cite{Milman}.

\vskip 25pt
\section*{Appendix A}
\vskip 10pt

\noindent \textbf{Lemma A.1.} \emph{Suppose $\{T_j\}_{j\in\mathbb N}\subset {\rm SL}(n)$. Then}
\[\|T_j\|\to\infty\quad \Longleftrightarrow\quad \|T_j^{-1}\|\to\infty. \]
\emph{Thus,  $\{T_j\}_{j\in\mathbb N}$ is bounded from above, if and only if  $\{T_j^{-1}\}_{j\in\mathbb N}$ is bounded from above.}

\begin{proof}
It suffices  to prove the implication
\begin{equation}\label{implication A.1}
\|T_j\|\to\infty\quad \Longrightarrow\quad \|T_j^{-1}\|\to\infty.
\end{equation}

For this aim, represent any $T\in {\rm SL(n)}$ in the form $ T=O_1AO_2$, where
$O_1, O_2$ are $n\times n$ orthogonal matrices, and
$A={\rm diag}(a_1,\cdots,a_n)$ is an $n\times n$ diagonal matrix, with
positive diagonal elements $a_1,\cdots, a_n$, and $\det(A)=1$. Then,
\begin{equation}\label{(A.1)}
\left\| T \right\|  = \mathop {\max }\limits_{1 \le i \le n} {a_i}.
\end{equation}

Observe that ${T^{ - 1}} = O_2^{ - 1}{A^{ - 1}}O_1^{ - 1}$, the
matrices $O_2^{-1}, O_1^{-1}$ are orthogonal, and $A^{-1}={\rm diag}\left(\frac{1}{a_1},\cdots,\frac{1}{a_n}\right)$.
Thus,
\begin{equation}\label{(A.2)}
\left\| {{T^{ - 1}}} \right\| = \mathop {\max }\limits_{1 \le i \le n} \frac{1}{{{a_i}}} = \frac{1}{{\mathop {\min }\limits_{1 \le i \le n} {a_i}}}.
\end{equation}

Meanwhile, the condition $\prod_{i=1}^n{a_i}=1$ together with the inequality
\[{\left( {\mathop {\min }\limits_{1 \le i \le n} {a_i}} \right)^{n - 1}}\mathop {\max }\limits_{1 \le i \le n} {a_i} \le \prod_{i=1}^n{a_i}\]
gives
\begin{equation}\label{(A.3)}
\frac{1}{{\mathop {\min }\limits_{1 \le i \le n} {a_i}}} \ge {\left( {\mathop {\max }\limits_{1 \le i \le n} {a_i}} \right)^{\frac{1}{{n - 1}}}}.
\end{equation}

Hence, by (\ref{(A.1)}), (\ref{(A.2)}) and (\ref{(A.3)}), the implication (\ref{implication A.1}) is derived.
\end{proof}

\noindent \textbf{Lemma A.2.} \emph{Suppose $\{T_j\}_{j\in\mathbb N}\subset {\rm SL}(n)$, $ T_0\in {\rm SL}(n)$. If  $T_j\to T_0$ with respect to  $d_n$, then}

\noindent \emph{(1)} $T_j^t B\to T_0^tB$ \emph{with respect to} $\delta_H$.

\noindent \emph{(2)} $T_j^{-1}\to T_0^{-1}$ \emph{with respect to} $d_n$.

\noindent \emph{(3)} $T_jB\to T_0 B$ \emph{with respect to} ${\tilde \delta}_H$.

\begin{proof}
From the following implications,
\begin{align*}
\left\| {{T_{j}} - {T_0}} \right\| \to 0\quad &\Longleftrightarrow\quad {T_{j}}u \to T_0u\;,  {\rm{uniformly}}\;{\rm{for}}\;u \in {S^{n - 1}}, \\
&\Longrightarrow\quad |{T_{j}}u| \to |T_0u|\;{\rm{uniformly}}\;, {\rm{for}}\;u \in {S^{n - 1}} ,\\
&\Longleftrightarrow\quad {h_{T_{j}^tB}}(u) \to {h_{T_0^tB}}(u)\;, {\rm{uniformly}}\;{\rm{for}}\;u \in {S^{n - 1}} ,\\
&\Longleftrightarrow\quad T_{j}^tB \to T_0^tB\;, {\rm with~ respect~ to}\; \delta_H,
\end{align*}
it yields  (1) directly.

Since $T_j\to T_0$, the sequence $\{T_j\}$ is bounded in
$(\mathscr{L}^n,d_n)$. By Lemma A.1, the sequence $\{T^{-1}_j\}$
is also bounded. Thus, to prove  $T_j^{-1}\to T_0^{-1}$, it
suffices to prove any convergent subsequence
$\{T^{-1}_{j_k}\}_{k\in\mathbb N}$ converges to $T^{-1}_0$.  Assume $T^{-1}_{j_k}\to T$.

Since $\|T_{j_k}-T_0\|\to 0$ and $\|T^{-1}_{j_k}-T\|\to 0$, from
$\mathop {\sup }\limits_{k \in \mathbb{N}} \left\| {{T_{{j_k}}}} \right\|<\infty$,  we have
\begin{align*}
\left\| {{T_{{j_k}}}T_{{j_k}}^{ - 1} - {T_0}T} \right\| &\le \left\| {{T_{{j_k}}}(T_{{j_k}}^{ - 1} - T)} \right\| + \left\| {({T_{{j_k}}} - {T_0})T} \right\| \\
&\le \left\| {T_{{j_k}}^{ - 1} - T} \right\|\left\| {{T_{{j_k}}}} \right\| + \left\| {{T_{{j_k}}} - {T_0}} \right\|\left\| T \right\| \\
&\le \left\| {T_{{j_k}}^{ - 1} - T} \right\|\mathop {\sup }\limits_{k \in \mathbb{N}} \left\| {{T_{{j_k}}}} \right\| + \left\| {{T_{{j_k}}} - {T_0}} \right\|\left\| T
\right\|,
\end{align*}
so, it concludes that $T_{j_k}T^{-1}_{j_k}\to T_0T$. Since $T_{j_k}T^{-1}_{j_k}=I_n$, $\forall k$,
it follows that $T=T^{-1}_0$.

That $T^{-1}_j\to T^{-1}_0$ with respect to $d_n$ implies that
\[|T_j^{ - 1}u| \to |T_0^{ - 1}u|;\quad {\rm i.e.},\quad \rho_{T_jB}(u)\to \rho_{T_0B}(u), \quad {\rm uniformly\; for}\; u\in S^{n-1}.\]
Thus, $T_jB\to T_0B$ with respect to  ${\tilde \delta}_H$.
\end{proof}

\noindent \textbf{Lemma A.3.} \emph{Suppose $E_0\in\mathcal{E}^n$,  $\{E_j\}_{j\in\mathbb N}\subset \mathcal{E}^n$ and $V(E_j)=a$,
$\forall j \in\mathbb N$,  $a>0$. Then $E_j\to E_0$ with respect to
$\delta_H$, if and only if $E_j\to E_0$ with respect to ${\tilde\delta}_H$.}

\begin{proof}
Under the standard orthonormal basis, there exist unique  symmetric
and positive definite matrices $T_0, T_j $, such that $T_0B=E_0$ and $T_j B=E_j$.

We first prove the following implication:
\begin{equation}\label{operator implication}
T_j B\to T_0B\;,  {\rm with ~respect ~to}\; \delta_H.  \quad \Longrightarrow\quad T_j\to T_0\;,  {\rm with ~respect ~to}\; d_n.
\end{equation}

That $T_jB\to T_0B$ with respect to $\delta_H$ implies that
$\sup\{\|T_j\|:j\in\mathbb N\}<\infty$. Thus, to prove $T_j\to T_0$,
it suffices to prove any convergent subsequence
$\{T_{j_k}\}_{k\in\mathbb N}$ of $\{T_j\}$ converges to $T_0$.
Assume $T_{j_k}\to T$. Let $T_j=\left(t^{j}_{l,m}\right)_{1\le l, m\le n}$ and
$T=\left(t_{l,m}\right)_{1\le l, m\le n}$.  By our  assumption, $x\cdot T_{j_k}y\to x\cdot Ty$, for $x\in\mathbb R^n$.

Now, three observations are in order. First,   $t^{j}_{l,m}\to t_{l,m}$,   $\forall (l,m)$.
Thus, the symmetry of each $T_j$ implies the symmetry of $T$;
Second, $\det(T)=1$. Indeed, since $\det(T_j)$ is
a continuous function of the elements of $T_j$, from the fact
$t^{j}_{l,m}\to t_{l,m}$ for each $(l,m)$, and the fact
$\det(T_j)=1$ for each $j$, we obtain $\det(T)=1$;  Third, $T$ is positive definite. Indeed, from
 $x\cdot T_jx\to x\cdot Tx$, together with the positive definitive of each
$T_j$, we know that  $T$ is positive semi-definite. Added that $\det(T)=1$, it follows that $T$ is positive definite.

Since $T_{j_k}\to T$, by Lemma A.2 (1), it follows that
$T_{j_k}B\to TB$ with respect to $\delta_H$. Thus, $TB=T_0B$. Since $T_0$ is the
unique symmetric positive definite matrix such that
$T_0 B=E_0$, it concludes that $T=T_0$. Thus, $T_j\to T_0$.

Now, from implication (\ref{operator implication}) and Lemma A.2 (1), it yields the implication
\[ E_j\to E_0\; {\rm with ~respect ~to}\;  \delta_H. \quad \Longrightarrow\quad E_j\to E_0\; {\rm with ~respect ~to}\;, {\tilde\delta}_H. \]

Conversely, assume that $E_j\to E_0$ with respect to ${\tilde \delta}_H$;
i.e., $\rho_{E_j} \to\rho_{E_0} $,  uniformly on $S^{n-1}$. From
(\ref{polarity1}) and the equation
\[|{h_{E_j^*}}(u) - {h_{E_0^*}}(u)| = \frac{{|{\rho _{{E_j}}}(u) - {\rho _{{E_0}}}(u)|}}{{{\rho _{{E_j}}}(u){\rho _{{E_0}}}(u)}},\quad {\rm for}\; u\in S^{n-1},\]
it follows that $h_{E_j^*}\to h_{E_0^*}$, uniformly on $S^{n-1}$,
i.e., $E_j^*\to E_0^*$ with respect to $\delta_H$. Note that $E^*_j=T^{-1}_jB$,
$E^*_0=T^{-1}_0B$, and $T^{-1}_j$, $T^{-1}_0$ are
both symmetric  positive definite. From implication (\ref{operator implication}),
it concludes that $T^{-1}_j\to T^{-1}_0$. Thus, By Lemma A.2 (2), $T_j\to T_0$. Therefore,
by Lemma A.2 (1), $T_jB\to T_0B$ with respect to $\delta_H$.
\end{proof}

\bibliographystyle{amsplain}

\end{document}